\documentclass[12pt]{article}
\usepackage{color}

\usepackage{amsmath}
\usepackage{amsfonts}
\usepackage{amsthm}
\usepackage{amssymb}
\usepackage{bm}
\usepackage[applemac]{inputenc}
\usepackage[mathcal,mathbf]{euler}
\usepackage{enumerate}
\makeatletter

\newtheoremstyle{mystyle}{}{}{\slshape}{2pt}{\scshape}{.}{ }{} 

\newtheorem{thm}{Theorem}[section]
\newtheorem{cor}[thm]{Corollary}

\newtheorem{prop}[thm]{Proposition}
\newtheorem{lem}[thm]{Lemma}
\newtheorem{fact}[thm]{Fact}

\newtheorem{conjecture}[thm]{Conjecture}

\newtheorem{problem}[thm]{Question}

\theoremstyle{definition}
\newtheorem{defi}[thm]{Definition}
\newtheorem{defn}[thm]{Definition}
\theoremstyle{mystyle}
\newtheorem{ex}[thm]{Example}

\theoremstyle{remark}
\newtheorem{rem}[thm]{Remark}

\newcommand{\monster}{\mathcal U}
\newcommand{\M}{\mathcal U}
\newcommand{\NIP}{NIP}
\newcommand{\Th}{Th}

\newcommand{\note}{\textbf}

\DeclareMathOperator{\tp}{tp}
\DeclareMathOperator{\Stab}{Stab}
\DeclareMathOperator{\FGen}{FGen}

\DeclareMathOperator{\Av}{Av}
\DeclareMathOperator{\Aut}{Aut}
\DeclareMathOperator{\Alt}{Alt}
\DeclareMathOperator{\ext}{ext}
\DeclareMathOperator{\mes}{\mathfrak{M}}
\DeclareMathOperator{\Def}{Def}
\DeclareMathOperator{\bdd}{bdd}
\DeclareMathOperator{\conv}{conv}
\DeclareMathOperator{\Inv}{Inv}

\def\indsym#1#2{%
 \setbox0=\hbox{$\m@th#1x$}%
 \kern\wd0%
 \hbox to 0pt{\hss$\m@th#1\mid$\hbox to 0pt{$\m@th#1^{#2}$\hss}\hss}%
 \lower.9\ht0\hbox to 0pt{\hss$\m@th#1\smile$\hss}%
 \kern\wd0}

\def\nindsym#1#2{%
 \setbox0=\hbox{$\m@th#1x$}%
 \kern\wd0%
 \hbox to 0pt{\hss$\m@th#1\not$\kern1.4\wd0\hss}
 \hbox to 0pt{\hss$\m@th#1\mid$\hbox to 0pt{$\m@th#1^{#2}$\hss}\hss}%
 \lower.9\ht0\hbox to 0pt{\hss$\m@th#1\smile$\hss}%
 \kern\wd0}

\newcommand*\samethanks[1][\value{footnote}]{\footnotemark[#1]}

\title{Definably amenable NIP groups}
\author{Artem Chernikov\thanks{The research leading to this paper has been partially supported by the European Research Council under the European Union's Seventh Framework Programme (FP7/2007-2013)/ERC Grant Agreement No. 291111 and by ValCoMo (ANR-13-BS01-0006).} \thanks{Partially supported by the Fondation Sciences Mathematiques de Paris (ANR-10-LABX-0098), by NSF (grants DMS-1600796 and DMS-1651321) and by the Sloan foundation}   and Pierre Simon\samethanks[1] \thanks{Partially supported by NSF (grant DMS-1665491), and by the Sloan foundation.}}

\begin{document}
\maketitle
\begin{abstract}
We study definably amenable NIP groups. We develop a theory of generics showing that various definitions considered previously coincide, and study invariant measures. As applications, we characterize ergodic measures, give a proof of the conjecture of Petrykowski connecting existence of bounded orbits with definable amenability in the NIP case, and prove the Ellis group conjecture of Newelski and Pillay connecting the model-theoretic connected component of an NIP group with the ideal subgroup of its Ellis enveloping semigroup.
\end{abstract}

\section{Introduction}

In the same way as algebraic or Lie groups are important in algebraic or differential geometry, the understanding of groups definable in a given first-order structure (or in certain classes of first-order structures) is important for model theory, as well as its applications.  On the one hand, even if one is only interested in abstract classification of first-order structures (i.e. in understanding combinatorial complexity of definable sets), unavoidably one is forced to study definable groups. (This realization probably started with Zilber's work on totally categorical structures \cite{zilberuncountably}, and later it was made clear by Hrushovski's theorem on unidimensional theories \cite{unidimensional}.) On the other hand, some of the most striking applications of model theory are based on a detailed understanding of definable groups in certain structures. The class of \emph{stable} groups is at the core of model theory, and the corresponding theory was developed in the 1970s-1980s borrowing many ideas from the study of algebraic groups over algebraically closed fields (with corresponding notions of connected components, stabilizers, generics, etc., see e.g. \cite{PoizatStableGroups}). In particular, this general theory was applied to  groups definable in differentially closed and separably closed fields, and was used by Hrushovski to prove  the Mordell-Lang conjecture for function fields \cite{MordellLang} . The theory of stable groups was generalized in the 90s to groups definable in a larger class of \emph{simple} theories, centered around the model-theoretic notion of forking (see \cite{WagnerBook}) and led to a number of results including Hrushovski's proof of the Manin-Mumford conjecture \cite{hrushovski2001manin} and other applications to algebraic dynamics (e.g. \cite{medvedev-scanlon}). More recently, inspired by the ideas of stable and simple group theory, Hrushovski has obtained a general stabilizer-type theorem and found striking applications to approximate subgroups \cite{hr_appx}, which led to a complete classification by Breuillard, Green and Tao \cite{breuillard2012}.
 On the other hand, groups definable in \emph{o-minimal structures} were studied extensively,  generalizing the theory of real Lie groups. This study culminated in a recent resolution of \emph{Pillay's conjecture} for compact o-minimal groups \cite{NIP1}, and the proof has brought to light the importance of the general theory of groups definable in NIP 
  structures (a common generalization of stable and o-minimal structures, see below) and the study of invariant measures on definable subsets of the group. In parallel, methods and objects of topological dynamics were introduced into the picture by Newelski \cite{New4} and gave rise to some new invariants coming from topological dynamics and conjectures concerning their relationship to the more familiar model-theoretic invariants. This circle of ideas has rapidly become a very active research area. The present paper contributes to this direction and, continuing the work in \cite{NIP1}, develops the theory of groups definable in NIP structures which admit a translation invariant probability measure on the boolean algebra of definable subsets. 

The NIP condition (the negation of the Independence Property) is a combinatorial tameness assumption on a first-order structure $M$ which says, in modern terms, that if $D\subseteq M^{m+n}$ is a definable set, then the family $\{ D(a)$, $a\in M^m \}$ of its fibers has finite VC-dimension (see Section \ref{sec:vc}). Roughly speaking, it says that the collection of definable subsets of $M$ is very structured. One can think of NIP as capturing the notion of a \emph{geometric} structure --- as opposed to say arithmetic or random-like --- and of NIP groups as groups arising in geometric settings. This condition was introduced by Shelah \cite{Sh10},  and the connection to VC-dimension was discovered later in \cite{LasVC}. In the past ten or fifteen years, the role of NIP theories has grown to become a central notion in model theory thanks to applications to o-minimal structures, valued fields and combinatorics (see e.g. \cite{SimBook} or \cite{starchenko2016nip} for a survey). Typical examples of NIP structures are given by stable structures (such as algebraically closed fields), o-minimal structures (such as real closed fields) and many Henselian valued fields. On the other hand, an ultraproduct of finite fields is an example of a structure which is not NIP, essentially because of arithmetic phenomena that enter the picture.

Let now $G$ be a group definable in an NIP structure $M$ (i.e. both the underlying set and multiplication are definable by formulas with parameters in $M$). Such a group comes equipped with a collection of \emph{definable} subsets of cartesian powers of $G$ which is closed under boolean combinations, projection and cartesian products. For example, if $M = (\mathbb{R}, +, \times, 0,1)$ is the field of reals, then $G$ is a real semi-algebraic group and definable sets are all semi-algebraic subsets. 
As it is typical in model theory, we prefer to work in a saturated model of our group (which in the case of an algebraic group corresponds to working in the universal domain in the sense of Weil). More precisely, let $\M \succeq M$ be a sufficiently saturated and homogeneous elementary extension of $M$, a ``monster model'' for the first-order theory of $M$.
We write $G(\M)$ to denote the group obtained by evaluating in $\M$ the formulas used to define $G$ in $M$ (and $G(M)$ will refer to the set of the $M$-points of $G$). So e.g. if we start with $M$ the field of reals, and $G(M)$ its additive group, then $G(U)$ is the additive group of a large real closed field extending $\mathbb{R}$ which now contains infinitesimals, infinitesimals relatively to those infinitesimals, etc. --- i.e., it satisfies a \emph{saturation} condition: every small enough finitely consistent family of definable sets has non-empty intersection.

It was shown by Shelah \cite{shelah2008minimal} that any NIP group $G(\M)$ admits a unique maximal compact quotient denoted $G/G^{00}$, which plays an important role in this theory and for which we will give a dynamical interpretation below.

Our goal in this paper is to adapt techniques from stable group theory to the NIP context in order to have tools at our disposal potentially as useful as those for stable groups. However, one main difference with the stable case is that we cannot deal with all groups any more. As we show, to have any hope of having well-behaved notions of generic types and large subsets, the group must be \emph{definably amenable} (this strengthens results of Hrushovski, Peterzil and Pillay \cite{NIP2} who first observed this).

We say that a definable group $G$ is definably amenable if there is a finitely-additive probability measure on the boolean algebra of definable subsets of $G$ which is moreover invariant under the group action (this property holds for $G(\M)$ if and only if it holds for $G(M)$, see the remark after Definition \ref{def: def amenable group}).
This notion has been introduced and studied in \cite{NIP1} and \cite{NIP2}. The emphasis in those papers is on the special case of the so-called \emph{fsg} groups, which will not be relevant to us here.
Of course, if $G(M)$ is amenable as a discrete group, then it is definably amenable since we have such a measure on all subsets, not just the definable ones, but the converse need not hold (e.g. deep work of Sela \cite{Sela-stable} demonstrates that any non-commutative free group, viewed as a first-order structure in the group language, is stable, hence definably amenable; but of course it is not amenable). Here are some important examples of definably amenable NIP groups:

$\bullet$ stable groups;

$\bullet$ definable compact groups in o-minimal theories or in p-adics ({\it e.g.} $SO_3(\mathbb R)$);

$\bullet$ solvable NIP groups, or more generally, any NIP group $G$ such that $G(M)$ is amenable as a discrete group.

Examples of definable NIP groups which are not definably amenable are $SL_2(\mathbb{R})$ or $SL_2(\mathbb{Q}_p)$ (see \cite{NIP1}).

It is classical in topological dynamics to consider the action of a discrete group $G$ on the compact space of ultrafilters on $G$, or in other words ultrafilters on the boolean algebra of \emph{all} subsets of $G$. In the definable setting, given a definable group $G(M)$, we let $S_G(M)$ denote the space of ultrafilters on the boolean algebra of \emph{definable} subsets of $G(M)$, hence the space $S_G(M)$ (called the  \emph{space of types} of $G(M)$) is a ``tame'' analogue of the Stone-\v{C}ech compactification of the discrete group $G$. Then $G(M)$ acts on $S_G(M)$ by homeomorphisms. The same construction applies to $G(\M)$ giving the space $S_G(\M)$ of ultrafilters on the definable subsets of $G(\M)$. Our main object of study in this paper are the dynamical systems $(S_G(M),G(M))$ and $(S_G(\M), G(\M))$ and related objects. In this context, we classify regular ergodic measures and show in particular that minimal flows are uniquely ergodic. We also give various characterizations of definable subsets of $G$ which have positive measure for some (resp. for all) invariant measures, connecting topological dynamics of the system with Shelah's model-theoretic notion of forking.

A starting point of this theory is a theorem of Shelah stating that any NIP group $G(\M)$ admits a maximal compact quotient $G/G^{00}$ (the kernel $G^{00}$ is characterized as the smallest subgroup of $G(\M)$ which is an intersection of definable subsets and has small index in $G(\M)$). We give a dynamical interpretation of this compact quotient by establishing an isomorphism between the ideal subgroup of the Ellis semigroup of a certain extension of $(S_G(M),G(M))$ and $G/G^{00}$. Those results settle several questions in the area.


\smallskip

~

Now we state the main results more precisely. In the case of stable groups, a natural notion of a generic set (or type) was given by Poizat (generalizing the notion of a generic point in an algebraic group), and a very satisfactory theory of such generics was developed in \cite{poizat1987groupes}. In a non-stable group, however, generic types need not exist, and several substitutes were suggested in the literature, either motivated by the theory of forking as in simple groups (\cite{NIP1, NIP2}), or by topological dynamics (\cite{NewPetr}). First we show that in a definably amenable NIP group all these notions coincide, and that in fact nice behaviour of these notions characterizes definable amenability.

\begin{thm}\label{thm: Main theorem 1} Let $G = G(\M)$ be a definable NIP group, with $\M$ a sufficiently saturated model. Then the following are equivalent:
\begin{enumerate}
\item $G$ is definably amenable (i.e. admits a $G$-invariant measure on its definable subsets).
\item The action of $G$ on $S_G(\M)$ admits a small orbit.
\end{enumerate}
\end{thm}

The proof is contained in Theorem \ref{thm: bdd orbit implies def amenable}. It  confirms a conjecture of Petrykowski in the case of NIP groups \cite[Conjecture 0.1]{New3} and solves Conjecture 4.13 of \cite{AnnalisaAnand}.

\begin{thm} \label{thm: Main theorem 2}Let $G = G(\M)$ be a definably amenable NIP group. Then the following are equivalent for a definable set $\phi(x)$:
\begin{enumerate}
\item $\phi(x)$ does not $G$-divide (i.e. there is no infinite sequence $(g_i)_{i < \omega}$ of elements of $G$ and natural number $k$ such that any $k$ sets in $\{ g_i \phi(x) \}_{ i< \omega}$ have empty intersection, see Definition \ref{def: f-gen});

\item $\phi(x)$ is weakly generic (i.e. there is some non-generic $\psi(x)$ such that $\phi(x) \lor \psi(x)$ is generic, see Definition \ref{def: weak generics, etc});
\item $\mu(\phi(x)) > 0$ for some $G$-invariant measure $\mu$;
\item $\phi(x)$ is $f$-generic (meaning that for any small model $M$ over which $\phi(x)$ is defined, no $G$-translate of $\phi(x)$ forks over $M$, see Definition \ref{def: f-gen});

\end{enumerate}
Moreover, for a global type $p \in S_G(\M)$ the following are equivalent:

\begin{enumerate}
\item $p$ is $f$-generic (i.e. every formula in $p$ is $f$-generic);
\item $p$ has a small $G$-orbit;
\item $\Stab(p) = G^{00}$.
\end{enumerate}

\end{thm}

This is given by Theorem \ref{thm:AllGenericNotionsCoincide} and Proposition \ref{prop:f-gen iff G00-inv}, and combined with Theorem \ref{thm: Main theorem 1} solves in particular \cite[Problem 4.13]{AnnalisaAnand}.

~

We continue by studying the space of $G$-invariant measures using VC-theory, culminating with a characterization of regular ergodic measures (Section \ref{sec: ergodicity}) and unique ergodicity (Section \ref{sec: unique ergodicity}). Generalizing slightly a construction from \cite{NIP2}, we associate to every generic type $p\in S_G(\M)$ a measure $\mu_p$, which is a lifting of the Haar measure on the compact group $G/G^{00}$ via $p$ (see Definition \ref{def: mu_p}). It follows from Theorem \ref{thm: Main theorem 2} that the supports of the measures $\mu_p$ are exactly the minimal subflows of $(S_G(\M),G(\M))$ (see Proposition \ref{prop: ap implies f-generic}).

\begin{thm}
Let $G = G(\M)$ be a definably amenable, NIP group. Then regular ergodic measures on $S_G(\M)$ are precisely the measures of the form $\mu_p$, for $p$ an $f$-generic type in $S_G(\M)$. If two such measures have the same support, then they are equal (i.e., minimal subflows of $(S_G(\M),G(\M))$ are uniquely ergodic).
\end{thm}

The first statement is Theorem \ref{thm: egodic iff mu_p} and the second follows from Proposition \ref{OrbitPreservesMuP}. The following is Theorem \ref{thm: unique ergodic}.

\begin{thm}
Let $G = G(\M)$ be a definably amenable, NIP group. Then $G$ has a unique invariant measure if and only if it admits a unique minimal subflow if and only if it admits a global generic type.
Moreover, in such a group all the notions in Theorem \ref{thm: Main theorem 2} coincide with ``$\phi(x)$ is generic'', and in the moreover part we can add ``$p$ is almost periodic''.
\end{thm}

Next we study enveloping semigroups. This notion from topological dynamics (see \cite{Glasner1}) was introduced in model theory by Newelski \cite{New4}. He observed that it behaved better when one replaced the dynamical system $(S_G(M),G(M))$ with an extension of it: The set $G(M)$ embeds into $S_G(\M)$ as realized types and we let $S_G(M^{\ext})$ be its closure. Then $G(M)$ acts on $S_G(M^{\ext})$ and this flow admits $S_G(M)$ as a factor. We consider the enveloping semigroup $E$ of the dynamical system $(S_G(M^{\ext}),G(M))$.  In view of the results in  \cite{ChePilSim}, $E$ can be identified with $(S_G(M^{\ext}), \cdot)$, where $\cdot$ is a naturally defined operation extending multiplication on $G(M)$ (see Section \ref{sec: Ellis group conj} for details).

Fix a minimal flow $\mathcal{M}$ in $(S_G(M^{\ext}),G(M))$ (i.e. a closed $G(M)$-invariant set), and an idempotent $u \in \mathcal{M}$. Then general theory of Ellis semigroups implies that  $u \mathcal{M}$ is a subgroup of $E$, which we call the Ellis group.
The canonical surjective homomorphism $G \to G/G^{00}$ factors naturally through the space $S_G(M^{\ext})$, so we have a well-defined continuous surjection $\pi: S_G(M^{\ext}) \to G/G^{00}, \tp(g/M) \mapsto gG^{00}$, and the restriction of $\pi$ to the group $u \mathcal{M}$ is a surjective homomorphism. Newelski asked if under certain model-theoretic assumptions this map could be shown to be an isomorphism. Pillay later formulated a precise conjecture which we are able to prove here.

\begin{thm}[Ellis group conjecture]
Let $G$ be definably amenable and NIP. Then $\pi: u \mathcal{M} \to G/G^{00}$ is an isomorphism.
\end{thm}

In particular, this demonstrates that the Ellis group is indeed a model theoretic object, i.e. it only depends on the first-order theory of the group and does not depend on the choice of a small model $M$ over which it is computed. Some special cases of the conjecture were previously known (see \cite{ChePilSim}). For the proof, we establish a form of generic compact domination for minimal flows in definably amenable groups with respect to the Baire ideal --- Theorem \ref{thm:CompactDomination}.

\begin{rem}

We remark that the study of NIP definably amenable groups can be thought of as a model-theoretic version of tame dynamics as studied by Glasner, Megrelishvili and others, see \cite{glasner2007structure} (in fact,  we discovered the connection only after having essentially completed this work). The NIP assumption implies that the dynamical system $(S_G(M),G(M))$ is tame---and even null---in the sense of \cite{glasner2007structure}, \cite{kerr2007independence}, but it is not equivalent to it. Nullness of this system is equivalent to the fact that the definable family of translates of any given definable set has finite VC-dimension (see \cite[Proposition 5.4(2)]{kerr2007independence}), whereas the NIP condition implies that any uniformly defined family of sets has finite VC-dimension.
\end{rem}

\subsection*{Acknowledgements}

We are grateful to the referees for pointing out a mistake in the previous version of the article (see Question \ref{prob: f-gen ideal vs amenability}) and for other numerous suggestions on improving the presentation.

%
%
%

\section{Preliminaries}

In this section we summarize some of the context for our results including the theory of forking and groups in NIP, along with some general results about families of sets of finite VC-dimension.

\subsection{Combinatorics of VC-families}\label{sec:vc}

Let $X$ be a set, finite or infinite, and let $\mathcal{F}$ be a family of subsets of $X$. 
Given $A \subseteq X$, we say that it is shattered by $\mathcal{F}$ if for every $A' \subseteq A$
there is some $S \in \mathcal{F}$ such that $A \cap S = A'$. A family $\mathcal{F}$ is said to have \emph{finite VC-dimension} if there is some $n < \omega$ such that
no subset of $X$ of size $n$ is shattered by $\mathcal{F}$. If this is the case, we let $VC(\mathcal{F})$ be the
largest integer $n$ such that some subset of $X$ of size $n$ is shattered by it.

If $S\subseteq X$ is a subset and $x_1,\ldots,x_n \in X$, we let $\Av(x_1,\ldots,x_n;S) = \frac 1 n |\{i\leq n: x_i\in S\}|$. Similarly, if $(t_i)_{i<n}$ is a set of truth values, we let $\Av(t_i) = \frac 1 n |\{i<n : t_i=$ True$\}$.

Later in the paper, we will often write $a \approx^{\epsilon} b$ for $|a-b|\leq \epsilon$.

A fundamental fact about families of finite VC-dimension is the following uniform version of the weak law of large numbers (\cite{vapnik1971uniform}, see also \cite[Section 4]{NIP2} for a discussion).
\begin{fact} \label{VC-theorem}
For any $k > 0$ and $\varepsilon > 0$ there is $N < \omega$ satisfying the following.

Let $(X, \mu)$ be a probability space, and  let $\mathcal{F}$ be a family of subsets of $X$ of VC-dimension $\leq k$ such that:
\begin{enumerate}
\item every set from $\mathcal{F}$ is measurable;
\item for each $n$, the function $f_n: X^n \to [0,1]$ given by $$(x_1, \ldots, x_n) \mapsto \sup_{S \in \mathcal{F}} | \Av(x_1, \ldots, x_n; S) - \mu(S) |$$  is measurable;
\item for each $n$, the function $g_n: X^{2n} \to [0,1]$
$$(x_1, \ldots, x_n, y_1, \ldots, y_n) \mapsto \sup_{S \in \mathcal{F}} | \Av(x_1, \ldots, x_n; S) - \Av(y_1, \ldots, y_n; S) |$$
is measurable.
\end{enumerate}
\underline{Then} there is some tuple $(x_1, \ldots, x_N) \in X^N$ such that for any $S \in \mathcal{F}$ we have $|\mu(S) - \Av(x_1, \ldots, x_N; S)| \leq \varepsilon$.
\end{fact}

The assumptions (2) and (3) are necessary in general (but follow from (1) if the family $\mathcal{F}$ is countable).

Another fundamental fact about VC-families that we will need is the following theorem about transversal sets due to Matousek. It uses the following definition: a family $\mathcal G$ of subsets of some set $X$ has the $(p,k)$- property if among any $p$ sets in $\mathcal G$, some $k$ have non-empty intersection.

\begin{fact}[\cite{Matousek}] \label{fac: p,q-theorem}
Let $\mathcal{F}$ be a family of subsets of some set $X$. Assume that $\mathcal{F}$ has finite VC-dimension. Then there is some $k < \omega$ such that for every $p \geq k$, there is an
integer $N$ such that: for every finite subfamily $\mathcal{G} \subseteq \mathcal{F}$, if $\mathcal{G}$ has the $(p, k)$-property, then there is an $N$-point
set intersecting all members of $\mathcal{G}$.
\end{fact}

\subsection{Forking in NIP theories} \label{sec: forking}

We will use standard notation. We work with a complete theory $T$ in a language $L$. We fix a monster model $\M \models T$
which is $\kappa$-saturated and $\kappa$-strongly homogeneous for $\kappa$ a sufficiently large strong limit cardinal.

Recall that a formula $\phi(x,y)$ is NIP if the family of subsets $\{\phi(x,a):a\in \monster\}$ has finite VC-dimension. The theory $T$ is NIP if all formulas are NIP. In this paper, we always assume that $T$ is NIP unless explicitly stated otherwise.

\smallskip

We summarize some facts about forking in NIP theories. Recall that a set $A$ is an \emph{extension base} if every type $p\in S(A)$ has a global extension non-forking over $A$. In particular, any model of an arbitrary theory is an extension base, and every set is an extension base in o-minimal theories, algebraically closed valued fields or p-adics.

\begin{defi}[\cite{CheKap}]
\begin{enumerate}
\item A global type $q \in S(\M)$ is \emph{strictly non-forking} over a small model $M$ if $q$ does not fork over $M$, and for every $B \supseteq M$ and $a \models q|_{B}$, $\tp(B/aM)$ does not fork over $M$.

\item Given $q \in S(M)$, we say that $(b_i : i < \kappa )$ is a \emph{strict Morley sequence} in $q$ if there is some global extension $q' \in S(\M)$ of $q$ strictly non-forking over $M$ satisfying $b_i \models q'|_{Mb_{<i}}$ for all $i < \kappa$.
\end{enumerate}

\end{defi}

\begin{fact}[\cite{CheKap}] \label{fac:ForkingNTP2} Assume that $T$ is NIP and let $A$ be an extension base.

\begin{enumerate}
\item A formula $\phi(x,a) \in L(\M)$ forks over $A$ if and only if it divides over $A$, i.e., the set of formulas dividing over $A$ forms an ideal.
\item Every $q(y) \in S(M)$ admits a global extension strictly non-forking over $M$.

\item Assume that $\phi(x,b) \in L(\M)$ forks (equivalently, divides) over $M$, and let $(b_i : i < \kappa)$ in $\M$ be an infinite strict Morley sequence in $\tp(b/M)$. Then $\{ \phi(x,b_i) : i<\kappa \}$ is inconsistent.

\end{enumerate}
\end{fact}




From now on, we will freely use the equivalence of forking and dividing over models in NIP theories.

\begin{fact}(See e.g. \cite[Proposition]{NIP2}.)
Assume that $T$ is NIP and $M \models T$. A global type $p(x)$ does not fork (equivalently, does not divide) over $M$ if and only if it is \emph{$M$-invariant}. This is, for every $\phi(x,a)$ and $a'\equiv_M a$, we have $p\vdash \phi(x,a)\Leftrightarrow p\vdash \phi(x,a')$. 
\end{fact}

\begin{rem}\label{rem: non-dividing extends to invariant}
In particular, in view of Fact \ref{fac:ForkingNTP2}, if $\pi(x)$ is a partial type that does not divide over $M$ (e.g. if $\pi(x)$ is $M$-invariant), then it extends to a global $M$-invariant type.
\end{rem}

Let now $p(x)$, $q(y)$ be global types invariant over $M$. For any set $D \supseteq M$, let $b \models q|_D, a \models p|_{Db}$. Then by invariance of $p$ and $q$, the type $\tp(ab/D)$ does not depend on the choice of $a,b$. Call this type $(p \otimes q)_D$, and let $p \otimes q = \bigcup \{ (p \otimes q)_D:  M \subseteq D \subseteq \M \mbox{ small}\}$. Then $(p \otimes q)(x,y)$ is a well-defined, global invariant type over $M$.

Let $p(x)$ be a global type invariant over $M$. Then one defines $$p^{(n)}(x_0, \ldots, x_{n-1}) = p(x_{n-1}) \otimes \cdots \otimes p(x_0),$$ $$ p^{(\omega)}(x_0, x_1, \ldots) = \bigcup_{n < \omega} p^{(n)}(x_0, \ldots, x_{n-1}).$$ For any small set $D \supseteq M$ and $(a_i)_{i < \omega} \models p^{(\omega)}|_D$, the sequence $(a_i)_{i < \omega}$ is indiscernible over $D$.

We now discuss Borel-definability. Let $p(x)$ be a global $M$-invariant type, pick a formula $\phi(x,y) \in L$, and consider the set $S_{p, \phi} = \{ a \in \M : \phi(x,a) \in p \}$. By invariance, this set is a union of types over $M$. In fact, it can be written as a finite boolean combination of $M$-type-definable sets (\cite{NIP2}). Specifically, let
 $\Alt_n (x_0, \ldots, x_{n-1}) = \bigwedge_{i < n - 1} \neg \left( \phi(x_i, y) \leftrightarrow \phi (x_{i+1}, y) \right)$ and let  $A_{n}(y)$ and $B_n(y)$ be the type-definable subsets of $\M$ defined by
$$\exists x_{0}\ldots x_{n-1}\left(p^{\left(n\right)}|_{M}\left(x_{0},\ldots,x_{n-1}\right) \land
\Alt_{n}(x_0, \ldots, x_{n-1})
\land \phi\left(x_{n-1}, y\right)\right)$$
and 
$$
\exists x_{0}\ldots x_{n-1}\left(p^{\left(n\right)}|_{M}\left(x_{0},\ldots,x_{n-1}\right)\land
\Alt_n (x_0, \ldots, x_{n-1})
\land\neg\phi\left(x_{n-1}, y\right)\right)$$

respectively.

Then for some $N<\omega$, $S_{p,\phi}=\bigcup_{n<N}\left(A_{n}\land\neg B_{n+1}\right)$.

Note that the set of all global $M$-invariant types is a closed subset of $S(\M)$. We now consider the local situation. Let $\phi(x,y) \in L$ be a fixed formula and let $S_{\phi}(\M)$ be the space of all global $\phi$-types (i.e., maximal consistent collections of formulas of the form $\phi(x,b), \neg \phi(x,b), b \in \M$). Let $\Inv_{\phi}(M)$ be the set of all global $M$-invariant $\phi$-types---a closed subset of $S_\phi(\M)$, which we equip with the induced topology.

\begin{fact}[\cite{Simon2014rosenthal}] \label{fac: seq compactness}
Let $M$ be a countable model and let $\phi(x,y)$ be NIP. For any set $Z \subseteq \Inv_\phi(M)$ and $p \in \Inv_\phi(M)$, if $p \in \overline{Z}$ (i.e., in the topological closure of $Z$), then $p$ is the limit of a countable sequence of elements of $Z$.
\end{fact}

\subsection{Keisler measures}\label{sec: Keisler measures}
Now we introduce some terminology and basic results around the study of measures in model theory. A \emph{Keisler measure} $\mu(x)$ (or $\mu_x$) over a set of parameters $A$ is a finitely additive probability measure
on the boolean algebra $\Def_x(A)$ of $A$-definable subsets of $\M$ in the variable $x$. Alternatively, a Keisler measure $\mu(x)$ may be viewed as assigning a measure to the clopen basis of the space of types $S_x(\M)$. A standard argument shows that it can be extended in a unique way to a countably-additive regular probability measure on all Borel subsets of $S_x(\M)$ (see e.g. \cite[Chapter 7]{SimBook} for details). From now on we will just say ``measure''  unless it could create some confusion.

For a measure $\mu$ over $A$ we denote by $S(\mu)$ its support: the set of types weakly random for $\mu$, i.e., the closed set of all $p\in S(A)$ such that for any $\phi(x)$, $\phi(x) \in p$ implies $\mu ( \phi(x) ) >0$. 

\begin{rem} \label{rem: TopOnMeas}Let $\mes_x(A)$ denote the set of measures over $A$ in variable $x$, it is naturally equipped with a compact topology as a closed subset of $[0,1]^{L_x(A)}$ with the product topology. Every type over $A$ can be identified with the $\{0,1\}$-measure concentrating on it, thus $S_x(A)$ is identified with a closed subset of $\mes_x(A)$.
\end{rem}

The following implication of Fact \ref{VC-theorem} was observed in \cite[Section 4]{NIP2}.

\begin{fact}
\label{fac: measure is average of types}Let $T$ be NIP. Let $\mu(x)$
a measure over $A$, $\Delta=\left\{ \phi_{i}\left(x,y_{i}\right)\right\} _{i<m}$
a finite set of $L$-formulas, and  $\varepsilon>0$ be arbitrary.
Then there are some types $p_{0},\ldots,p_{n-1}\in S_x\left(A\right)$
such that for every $a\in A$ and $\phi\left(x,y\right)\in\Delta$,
we have 
\[
\left|\mu\left(\phi\left(x,a\right)\right)-\Av\left(p_{0},\ldots,p_{n-1};\phi\left(x,a\right)\right)\right|\leq\varepsilon\mbox{.}
\]
Furthermore, we may assume that $p_{i}\in S\left(\mu\right)$, the
support of $\mu$, for all $i<n$.\end{fact}

\begin{cor}
Let $T$ be an NIP theory in a countable language $L$, and let $\mu$ be a measure. Then the support $S(\mu)$ is separable (with respect to the topology induced from $S(\M)$).
\end{cor}
\begin{proof}
By Fact \ref{fac: measure is average of types}, for any finite $\Delta \subseteq L$ and $k < \omega$, we can find some $p_0^{\Delta}, \ldots, p_{n^\Delta_k - 1}^{\Delta} \in S(\mu)$ such that for any $\phi(x,y) \in \Delta$ and any $a \in \M$ we have $\mu(\phi(x,a)) \approx^{\frac{1}{k}} \Av(p_0^{\Delta}, \ldots, p_{n^{\Delta}_k - 1}^{\Delta}; \phi(x,a))$. Let $S_0 = \bigcup_{k < \omega, \Delta \subseteq L \textrm{ finite}} \{ p^{\Delta}_i : i < n^{\Delta}_k\}$. Then $S_0$ is a countable subset of $S(\mu)$, and we claim that it is dense. Let $U$ be a non-empty open subset of $S(\mu)$. Then there is some formula $\phi(x) \in L(\M)$ such that $\emptyset \neq \phi(x) \cap S(\mu) \subseteq U$. In particular $\mu(\phi(x)) > 0$, hence for some $k$ and $\Delta$ large enough we have by the construction of $S_0$ that necessarily $\phi(x) \in p_i^{\Delta}$ for at least one $i < n^{\Delta}_k$.
\end{proof}

A measure $\mu \in \mes_x(\M)$ is non-forking over a small model $M$ if for every formula $\phi(x) \in L(\M)$ with $\mu(\phi(x))>0$, $\phi(x)$ does not fork over $M$. A theory of forking for measures in NIP generalizing the previous section from types to measures is developed in \cite{NIP2, HruPilSimMeas}. In particular, a global measure non-forking over a small model $M$ is in fact $\Aut(\M/M)$-invariant. Moreover, using Fact \ref{fac: measure is average of types} along with results in Section \ref{sec: forking} one shows that a global measure $\mu$ invariant over $M$ is \emph{Borel definable} over $M$, i.e., for any $\phi(x,y) \in L$ the map $f_\phi: S_y(M) \to [0,1], q \mapsto \mu(\phi(x,b)), b \models q$ is Borel (and it is well defined by $M$-invariance of $\mu$). This allows to define a tensor product of $M$-invariant measures: given $\mu \in \mes_x(\M), \nu \in \mes_y(\M)$ $M$-invariant and $\phi(x,y) \in L(\M)$, let $N\supseteq M$ be some small model over which $\phi$ is defined. We define $\mu \otimes \nu(\phi(x,y))$ by taking $\int_{q \in S_y(N)} f_\phi(q) \textrm{d} \nu'$, where $\nu' = \nu|_N$ viewed as a Borel measure on $S_y(N)$. Then $\mu \otimes \nu$ is a global $M$-invariant measure.

~

We will need the following basic combinatorial fact about measures (see \cite{NIP1} or \cite[Lemma 7.5]{SimBook}).
\begin{fact}\label{fact_indmeasures}
Let $\mu$ be a Keisler measure, $\phi(x,y)$ a formula and $(b_i)_{i<\omega}$ an indiscernible sequence. Assume that for some $\epsilon>0$ we have $\mu(\phi(x,b_i))\geq \epsilon$ for every $i<\omega$. Then the partial type $\{\phi(x,b_i):i<\omega\}$ is consistent.
\end{fact}

\subsection{Model-theoretic connected components} \label{sec: connected components}

Now let $G = G(\M)$ be a definable group. Let $A$ be a small subset of $\M$. We say that $H \leq G$ has \emph{bounded} index if $|G:H|$ is smaller than the saturation of $\M$, and define:
\begin{itemize}
\item $G_{A}^{0}=\bigcap\left\{ H\leq G:H\mbox{ is }A\mbox{-definable, of finite index}\right\} $.
\item $G_{A}^{00}=\bigcap\left\{ H\leq G:H\mbox{ is type-definable over \ensuremath{A}, of bounded index}\right\} $.
\item $G_{A}^{\infty}=\bigcap\left\{ H\leq G:H\mbox{ is \ensuremath{\Aut\left(\M/A\right)}-invariant, of bounded index}\right\} $.
\end{itemize}
Of course $G_{A}^{0}\supseteq G_{A}^{00}\supseteq G_{A}^{\infty}$ for any $A$ and these are all normal $A$-invariant subgroups of $G$.

\begin{fact}[see e.g. {\cite[Chapter 8]{SimBook}} and references therein]
Let $T$ be $\NIP$. Then for every small set $A$ we have $G_{A}^{0}=G_{\emptyset}^{0}$, $G_{A}^{00}=G_{\emptyset}^{00}$, $G_{A}^{\infty}=G_{\emptyset}^{\infty}$. Moreover, $|G/G^{\infty}| \leq 2^{|T|}$.
\end{fact}

We will be omitting $\emptyset$ in the subscript and write for instance $G^{00}$ for $G^{00}_\emptyset$.
\begin{rem}\label{rem: G000 generators}
It follows that $G^{\infty}$ is equal to the subgroup of $G$ generated by the set $\{ g^{-1} h : g \equiv_M h \}$, for any small model $M$.
\end{rem}

Let $\pi : G \to G/G^{00}$ be the canonical projection map.

The quotient $G/G^{00}$ can be equipped with
a natural ``logic'' topology: a set $S\subseteq G/G^{00}$ is closed iff $\pi^{-1}\left(S\right)$ is type-definable over some (equivalently, any) small model $M$.

\begin{fact} [see \cite{PillayLogicTop}]
The group $G/G^{00}$ equipped with the logic topology is a compact topological group.
\end{fact}

\begin{rem} \label{rem: G/G^{00} is Polish}
If $L$ is countable then $G/G^{00}$ is a Polish space with respect to the logic topology. Indeed, there is a countable model $M$ such that every closed set is a projection of a partial type over $M$, and $\{ \pi(\phi(\M))^c : \phi(x) \in L(M)\}$ is a countable basis of the topology.
\end{rem}

In particular, $G/G^{00}$ admits an invariant normalized Haar probability measure $h_0$. Furthermore $h_0$ is the unique left-$G/G^{00}$-invariant Borel probability measure on $G/G^{00}$ (see e.g. \cite[Section 60]{halmos1950measure}), as well as simultaneously the unique right-$G/G^{00}$-invariant Borel probability measure on $G/G^{00}$.

The usual completion procedure for a measure preserves $G$-invariance, so we may take $h_0$ to be complete.

\section {Generic sets and measures}

\subsection{$G$-dividing, bounded orbits and definable amenability}
\textbf{Context:} We work in an NIP theory $T$, and let $G = G(\M)$ be an $\emptyset$-definable group.

We will consider $G$ as acting on itself on the left. For any model $M$, this action extends to an action of $G(M)$ on the space $S_G(M)$ of types concentrating on $G$. Hence if $p\in S_G(M)$ and $g\in G(M)$ we have $g\cdot p = \tp(g\cdot a/M)$ where $a\models p$. The group $G(M)$ also acts on $M$-definable subsets of $G$ by $(g\cdot \phi)(x) = \phi(g^{-1}\cdot x)$ and on measures by $(g\cdot \mu)(\phi(x))=\mu(\phi(g \cdot x))$.

One could also consider the right action of $G$ on itself and obtain corresponding notions. Contrary to the theory of stable groups, this would not yield equivalent definitions (see Section \ref{sec_leftright} for a discussion).

\begin{defi}\label{def: def amenable group}
The group $G$ is \emph{definably amenable} if it admits a global Keisler measure $\mu$ on definable subsets of $G(\M)$ which is invariant under (left-) translation by elements of $G(\monster)$.
\end{defi}

As explained for example in \cite[8.2]{SimBook}, if for some model $M$, there is a $G(M)$-invariant Keisler measure on $M$-definable subsets of $G$, then $G$ is definably amenable (it can be seen by taking an elementary extension $M$ expanded by predicates for the invariant measure).

\begin{defi} \label{def: f-gen}
\begin{enumerate}
\item Let $\phi(x)$ be a subset of $G$ defined over some model $M$. We say that $\phi(x)$ (left-)\emph{$G$-divides} if there is an $M$-indiscernible sequence $(g_i:i<\omega)$ such that $\{g_i \cdot \phi(x) :i<\omega\}$ is inconsistent.

\item The formula $\phi(x)$ is (left-)\emph{$f$-generic over $M$} if no translate of $\phi(x)$ forks over $M$. We say that $\phi(x)$ is \emph{$f$-generic} if it is $f$-generic over some small $M$. A (partial) type is $f$-generic if every formula implied by it is $f$-generic.

\item A global type $p$ is called (left-)\emph{strongly f-generic} over $M$ if no $G(\M)$-translate of $p$ forks over $M$. A global type $p$ is strongly $f$-generic if it is strongly $f$-generic over some small model $M$.

\end{enumerate}
\end{defi}

Note that we change the usual terminology: our notion of strongly $f$-generic corresponds to what was previously called $f$-generic in the literature (see e.g. \cite{NIP2}). We feel that this change is justified by the development of the theory presented here.

Note that if $\mu$ is a global $G$-invariant and $M$-invariant measure and $p\in S(\mu)$, then $p$ is strongly f-generic over $M$ since all its translates are weakly-random for $\mu$. It is shown in \cite{NIP2} how to conversely obtain a measure $\mu_p$ from a strongly f-generic type $p$. We summarize some of the results from \cite{NIP2} in the following fact.

Recall that the stabilizer of $p$ is $\Stab_G(p)=\{g\in G : g\cdot p = p\}$.

\begin{fact}\label{fac: BasicDefAm}
\begin{enumerate}
\item If $G$ admits a strongly $f$-generic type over some small model $M$, then it admits a strongly $f$-generic type over any model $M_0$.

\item If $p$ is strongly $f$-generic then $\Stab_G(p) = G^{00} = G^{\infty} ( = \langle \{ g^{-1} h : g \equiv_M h \} \rangle$ for any small model $M$).

\item The group $G$ admits a $G$-invariant measure if and only if there is a global strongly $f$-generic type in $S_G(\M)$.
\end{enumerate}
\end{fact}

Our first task is to understand basic properties of f-generic formulas and types.

\begin{prop}\label{prop_Gdivide} Let $G$ be a definably amenable group, and let $\phi(x)\in L_G(M)$. Let also $p(x)\in S_G(\monster)$ be strongly f-generic, $M$-invariant and take $g\models p|_M$. Then the following are equivalent:

1. $\phi(x)$ is f-generic over $M$;

2. $\phi(x)$ does not $G$-divide;

3. $g^{-1}\cdot \phi(x)$ does not fork over $M$.
\end{prop}
\begin{proof}
$(2) \Rightarrow (1)$: Assume that some translate $h\cdot \phi(x)$ forks over $M$. Then it divides over $M$, and as $\phi(x)$ is over $M$, we obtain an $M$-indiscernible sequence $(h_i:i<\omega)$ such that $\{h_i \cdot \phi(x):i<\omega\}$ is inconsistent. This shows that $\phi(x)$ $G$-divides.


$(1) \Rightarrow (3)$: Clear.

$(3) \Rightarrow (2)$: Assume that $\phi(x)$ does $G$-divide and let $(g_i:i<\omega)$ be an $M$-indiscernible sequence witnessing it, i.e., $\{ g_i \cdot \phi(x) : i < \omega \}$ is $k$-inconsistent for some $k < \omega$. By indiscernibility, all of $g_{i}$'s are in the same $G^{00}$-coset, and replacing $g_i$ by $g_0^{-1}g_{i+1}$, we may assume that $g_{i} \in G^{00}$ for all $i$.

Let $h$ realize $p$ over $(g_i)_{i < \omega} M$. Then $g_i^{-1}\cdot h \models p|_M$ by $G^{00}$-invariance of $p$. As the set $\{g_i\cdot \phi(x):i<\omega\}$ is inconsistent, so is $\{h^{-1}g_i \cdot \phi(x):i<\omega\}$. Then the sequence $(g_i^{-1}\cdot h:i<\omega)$ is an $M$-indiscernible sequence in $p|_M = \tp(g/M)$ (as $\tp(h/(g_i)_{i < \omega} M)$ is $M$-invariant). Therefore $g^{-1}\cdot \phi(x)$ divides over $M$.
\end{proof}

Note that we do not say ``$G$-divides over $M$", because the model $M$ does not matter in the definition: for any $M \prec N$, an $M$-definable $\phi(x)$ $G$-divides over $M$ if and only if it $G$-divides over $N$. Therefore the same is true for $f$-genericity (i.e. if $\phi(x)$ is both $M$-definable and $N$-definable, then it is $f$-generic over $M$ if and only if it is $f$-generic over $N$) and from now on we will just say $f$-generic, without specifying the base.

\begin{cor}\label{prop_ideal}
Let $G$ be definably amenable. The family of non-f-generic formulas (equivalently, $G$-dividing formulas) forms an ideal. In particular, every partial $f$-generic type extends to a global one.
\end{cor}
\begin{proof}
Assume that $\phi(x),\psi(x)$ are not $f$-generic, and let $M$ be some small model over which both formulas are defined. Let also $p$ be a global type strongly $f$-generic over $M$ (exists by Fact \ref{fac: BasicDefAm}) and take $g \models p|_M$. Then by Fact \ref{prop_Gdivide}(3) we have that both $g^{-1} \cdot \phi(x), g^{-1} \cdot \psi(x)$ fork over $M$, in which case $g^{-1} \cdot (\phi(x) \lor \psi(x)) = g^{-1} \cdot \phi(x) \lor g^{-1}\cdot \psi(x)$ also forks over $M$. Applying Fact \ref{prop_Gdivide}(3) again it follows that $\phi(x) \lor \psi(x)$ is not $f$-generic.

The ``in particular" statement follows by compactness.
\end{proof}


\begin{lem} \label{NonGDivFormulaG00Inv}
Let $G$ be definably amenable, let $\phi(x) \in L_G(\M)$ be a formula and $g\in G^{00}$. Then $\phi(x)\triangle g\cdot \phi(x)$ is not $f$-generic (and hence it $G$-divides by Proposition \ref{prop_Gdivide}).
\end{lem}
\begin{proof}
Let $M$ be a model over which $\phi(x)$ and $g$ are defined. Let $p\in S_G(\monster)$ be a global strongly f-generic type which is $M$-invariant (exists by Fact \ref{fac: BasicDefAm}(1)) and let $h$ realize $p$ over $Mg$. Then $h^{-1}\cdot (\phi(x)\triangle g\cdot \phi(x)) = (h^{-1}\cdot \phi(x)) \triangle (h^{-1}g\cdot \phi(x))$. Since $h \equiv_M g^{-1}h$ (as $g^{-1} \in \Stab_G(p)$ by Fact \ref{fac: BasicDefAm}(2)), the  latter formula cannot belong to any global $M$-invariant type, and so it must fork over $M$ by Remark \ref{rem: non-dividing extends to invariant}. Hence $\phi(x)\triangle g\cdot \phi(x)$ is not $f$-generic.
\end{proof}

\begin{defi} A global type $p(x) \in S_G(\M)$ has a \emph{bounded orbit} if $| G\cdot p |<\kappa$ for some strong limit cardinal $\kappa$ such that $\M$ is $\kappa$-saturated.
\end{defi}

\begin{prop}\label{prop:f-gen iff G00-inv}
Let $G$ be definably amenable. For $p\in S_G(\monster)$, the following are equivalent:
\begin{enumerate}
\item $p$ is f-generic,
\item $p$ is $G^{00}$-invariant (and $\Stab_G(p) = G^{00}$),
\item $p$ has a bounded orbit.
\end{enumerate}
\end{prop}
\begin{proof}
(1) $\Rightarrow$ (2): If $p$ is not $G^{00}$-invariant then $\phi(x) \triangle g\phi(x) \in p$ for some $g \in G^{00}, \phi(x) \in L_G(\M)$, and so $p$ is not $f$-generic by Lemma \ref{NonGDivFormulaG00Inv}. 
Hence $G^{00} \subseteq \Stab_G(p)$. Given an arbitrary $a \in \Stab_G(p)$, let $M$ be a small model containing $a$ and let $b \models p|_{M}$. Then $a\cdot b \models p|_{M}$, hence $a = (a\cdot b) \cdot b^{-1}$ and $a \cdot b \equiv_M b$. By Fact \ref{fac: BasicDefAm}(2) it follows that $a \in G^{00}$, hence $\Stab_G(p) = G^{00}$.

(2) $\Rightarrow$ (3): If $p$ is $G^{00}$-invariant, then the size of its orbit is bounded by the index of $G^{00}$ (which is $\leq 2^{|T|}$).

(3) $\Rightarrow$ (1): If $p$ is not $f$-generic, then some $\phi(x) \in p$ must $G$-divide (by Proposition \ref{prop_Gdivide}). Then, as in the proof of Proposition \ref{prop_Gdivide}, we can find an arbitrarily long indiscernible sequence $(g_i)_{i<\lambda}$ in $G^{00}$ such that $\{ g_i\phi(x) : i < \lambda \}$ is $k$-inconsistent for some $k< \omega$, which implies that the $G$-orbit of $p$ is unbounded.
\end{proof}

Next we clarify the relationship between $f$-generic and strongly $f$-generic types in definably amenable groups.

\begin{prop}\label{prop: f-gen iff inv and bdd orbit}
Let $G$ be definably amenable. A type $p\in S_G \left(\M\right)$ is strongly $f$-generic
if and only if it is $f$-generic and $M$-invariant over some small model $M$.\end{prop}
\begin{proof}
Strongly $f$-generic implies $f$-generic is clear.

Conversely, assume that $p$ is $M$-invariant, but not strongly $f$-generic
over $M$. Then $g\cdot p$ divides over $M$ for some $g\in G$. It follows
that there is some $\phi\left(x,a\right)\in p$ such that for any
$\kappa$ there is some $M$-indiscernible sequence $\left(g_{i}\hat{~}a_{i}\right)_{i<\kappa}$
with $g_{0}\hat{~}a_{0}=g\hat{~}a$ and such that $\left\{ g_{i}\cdot\phi\left(x,a_{i}\right)\right\} _{i<\kappa}$
is $k$-inconsistent for some $k < \omega$. By $M$-invariance of $p$ we have that $\phi\left(x,a_{i}\right)\in p$,
so $\left\{ g_{i}\cdot p\left(x\right)\right\} _{i<\kappa}$ is
$k$-inconsistent. This implies that the orbit of $p$ is unbounded, and that $p$ is not $f$-generic in view of Proposition \ref{prop:f-gen iff G00-inv}.\end{proof}

\begin{ex}
There are $f$-generic types which are not strongly $f$-generic. Let $\mathcal R$ be a saturated model of RCF. We give an example of a $G$-invariant (and so $f$-generic by Proposition \ref{prop:f-gen iff G00-inv}) type in $G = (\mathcal R^2;+)$ which is not invariant over any small model (and so not strongly $f$-generic by Proposition \ref{prop: f-gen iff inv and bdd orbit}). Let $p(x) \in S_1(\mathcal{R})$ denote the definable 1-type at $+\infty$ and $q(y) \in S_1(\mathcal{R})$ a global 1-type which is not invariant over any small model (hence corresponds to a cut of maximal cofinality from both sides). Then $p$ and $q$ are weakly orthogonal types. Let $(a,b)\models p\times q$ (in some bigger model) and consider $r :=\tp(a,a + b/\mathcal R)$. Then $r \in S_G(\mathcal{R})$ is a $G$-invariant type which is not invariant over any small model.
\end{ex}

The following lemma is standard.
\begin{lem}
\label{lem: extending f-generic types}
Let $N \succ M$ be $|M|^+$-saturated, and let $p \in S_G(N)$ be such that $g \cdot p$ does not fork over $M$ for every $g \in G(N)$. Then $p$ extends to a global type strongly $f$-generic over $M$.
\end{lem}
\begin{proof}
It is enough to show that $$p(x) \cup \left\{ \neg(g\cdot \phi(x,a)) : g \in G(\M), \phi(x,a) \in L(\M) \mbox{ forks over M} \right\}$$ is consistent. Assume not, then $p(x) \vdash \bigvee_{i<n} g_i \cdot \phi_i (x,a_i)$ for some $g_i \in G(\M)$, $\phi_i(x,y) \in L$ and $a_i \in \M$ such that $\phi_i(x,a_i)$ forks over $M$. By $|M|^+$-saturation of $N$ and compactness we can find some $(g'_i,a'_i )_{i<n} \equiv_M (g_i,a_i)_{i<n}$ in $N$ such that $p(x) \vdash \bigvee_{i<n} g'_i \cdot \phi_i (x,a'_i)$, which implies that $g'_i \cdot \phi_i(x,a'_i) \in p$ for some $i<n$, i.e., $(g'_i)^{-1} \cdot p$ forks over $M$. But this contradicts the assumption on $p$.
\end{proof}

Finally for this subsection, we prove Theorem \ref{thm: Main theorem 1}: for NIP groups, definable amenability is characterized by the existence of a type with a bounded orbit, proving Petrykowski's conjecture for $\NIP$ theories (see \cite[Conjecture 0.1]{New3}). In fact, existence of a measure with a bounded orbit is sufficient.
\begin{thm}
\label{thm: bdd orbit implies def amenable}Let $T$ be $\NIP$, $\M\models T$ and $G = G(\M)$ a definable group. Then the following are equivalent:
\begin{enumerate}
\item $G$ is definably amenable;
\item $|G\cdot p| \leq 2^{|T|}$ for some $p \in S_G(\M)$;
\item some measure $\mu \in \mes_G(\M)$ has a bounded $G$-orbit.
\end{enumerate}
\end{thm}
\begin{proof}

(1) $\Rightarrow$ (2): If $G$ is definably amenable, then it has a strongly $f$-generic type $p \in S_G(\M)$ by Fact \ref{fac: BasicDefAm} and such a type is $G^{00}$-invariant. In particular its orbit has size at most $|G/G^{00}| \leq 2^{|T|}$.

(2) $\Rightarrow$ (3) is obvious.

(3) $\Rightarrow$ (1): Assume that $\left|G \mu \right|<\kappa$, with $\kappa$ strong limit and $\M$ is $\kappa$-saturated. Let
$M$ be a model with $|M|=|T|$, let $N_{0}\succ M$ be an $|M|^+$-saturated submodel of $\M$ of size $2^{|M|} < \kappa$ (exists as $\kappa$ is a strong limit cardinal), and let $\left(N_{i}\right)_{i<\kappa}$
be a strict Morley sequence in $\tp\left(N_{0}/M\right)$ contained in $\M$ (exists by $\kappa$-saturation of $\M$ and Fact \ref{fac:ForkingNTP2}(2)). In particular
$N_{i}$ is an $|M|^+$-saturated extension of $M$ for all $i < \kappa$. 

Let $\mu_i = \mu| _{N_i}$. It is enough to show that for some $i < \kappa$, the measure $g \mu_i$ does not fork over $M$ for any $g \in G(N_i)$, as then any type in the support of $\mu_i$ extends to a global type strongly $f$-generic over $M$ by Lemma \ref{lem: extending f-generic types}, and we can conclude by Fact \ref{fac: BasicDefAm}.

Assume not, then for each $i< \kappa$ we have some $g_{i}\in G\left(N_{i}\right)$ and some $\phi_i(x,c_i) \in L(N_i)$
such that $g_i \mu_i (\phi(x,c_i)) > 0$ but $\phi_i(x,c_i)$ forks over $M$. 

As the orbit of $\mu$ is bounded, by throwing away some $i$'s we may assume that there is some $g \in G$ such that $g_i \mu = g \mu$ for all $i< \kappa$, in particular $(g\mu)|_{N_i} = g_i \mu_i$. By pigeonhole and the assumption on $\kappa$ we may assume also that there are some $\phi(x,y) \in L$ and $\varepsilon > 0$ such that $\phi_i(x,y_i) = \phi(x,y)$ and $g \mu(\phi(x,c_i)) > \varepsilon$ for all $i < \kappa$, and that the sequence $(c_i: i < \kappa)$ is indiscernible (i.e. the $c_i$'s occupy the same place in the enumeration of $N_i$, for all $i$, and the sequence $(N_i)_{i< \kappa}$ is indiscernible by construction). Applying Fact \ref{fact_indmeasures} to the measure $g \mu$ we conclude that $\{ \phi(x,c_i) : i < \kappa \}$ is consistent. But as $(c_i)$ is a strict Morley sequence, 
this contradicts the assumption that $\phi(x,c_i)$ divides over $M$ for all $i$, in view of Fact \ref{fac:ForkingNTP2}(3).\end{proof}

\begin{rem}
\begin{enumerate}
\item In the special case of types in $o$-minimal expansions of real closed fields this was proved in \cite[Corollary 4.12]{AnnalisaAnand}.

\item Theorem \ref{thm: bdd orbit implies def amenable} also shows that the issues with absoluteness of the existence of a bounded orbit considered in \cite{New3} do not arise when one restricts to NIP groups.
\end{enumerate}
\end{rem}

\subsection{Measures in definably amenable groups} \label{sec: ErgodicMeasures}

\subsubsection{Construction}
Again, we are assuming throughout this section that $G = G(\M)$ is an NIP group. We generalize the connection between $G$-invariant measures and strongly $f$-generic types from Fact \ref{fac: BasicDefAm} to $f$-generic types in definably amenable groups.

%
%

First we generalize Proposition \ref{prop:f-gen iff G00-inv} to measures.
\begin{prop} \label{lem_PosGInvMeasImpliesFGen} Let $G$ be definably amenable, and  let $\mu$ be a Keisler measure on $G$. The following are equivalent:
\begin{enumerate}
\item The measure $\mu$ is $f$-generic, that is $\mu(\phi(x)) > 0 $ implies $\phi(x)$ is $f$-generic for all $\phi(x) \in L_G(\M)$.
\item All types in the support $S(\mu)$ are $f$-generic.
\item The measure $\mu$ is $G^{00}$-invariant.
\item The orbit of $\mu$ is bounded.
\end{enumerate}
\end{prop}
\begin{proof}
The equivalence of (1) and (2) is clear by compactness, (1) implies (3) is immediate by Lemma \ref{NonGDivFormulaG00Inv}, and (3) implies (4) as the size of the orbit of a $G^{00}$-invariant measure is bounded by $|G/G^{00}|$.

$(4) \Rightarrow (1)$: Assume that we have some $G$-dividing $\phi(x)$ with $\mu(\phi(x)) > \varepsilon > 0$. As in the proof of Proposition \ref{prop_Gdivide}, (3) $\Rightarrow$ (2) we can find an arbitrarily long indiscernible sequence $(g_i)_{i \in \lambda}$ with $g_i \in G^{00}$ such that $\{ g_i \phi(x) \}$ is $k$-inconsistent, for some fixed $k < \omega$. 

In view of Fact \ref{fact_indmeasures} for any fixed $i < \lambda$ there can be only finitely many $j < \lambda$ such that $g_i \mu (g_j \phi(x)) > \varepsilon$. But $g_i \mu(g_j \phi(x)) = g_j^{-1} g_i \mu(\phi(x))$. This implies that $g_i \mu \neq g_j \mu$ for all but finitely many $j < \lambda$, which then implies that the orbit of $\mu$ is unbounded.
\end{proof}




In \cite[Proposition 5.6]{NIP2} it is shown that one can lift the Haar measure on $G/G^{00}$ to a global $G$-invariant measure on all definable subsets of an NIP group $G$ using a strongly $f$-generic type. We point out that in a definably amenable NIP group, an $f$-generic type works just as well. For this we need a local version of the argument used there. 

Fix a small model $M$, and let $\mathcal{F}_{M}$ be the set of formulas of the form $g\cdot\phi\left(x\right)$ or $\neg g\cdot \phi(x)$, for $g\in G\left(\M\right),\phi\left(x\right)\in L_G\left(M\right)$.


\begin{prop} \label{prop: f-gen iff M-inv, locally}
Let $G$ be definably amenable, and let $p$ be a maximal finitely consistent set of formulas in $\mathcal{F}_{M}$.
Then $p$ is $f$-generic if and only if $g\cdot p$ is $M$-invariant
for every $g\in G$.\end{prop}
\begin{proof}
Notice that $g\cdot p(x)$ is also a set of formulas in $\mathcal F_M$. Assume that $g\cdot p\left(x\right)$ is not $M$-invariant.
Then $gp\vdash g_{0}\phi\left(x\right)\triangle g_{1}\phi\left(x\right)$
for some $\phi\left(x\right)\in L\left(M\right)$ and $g_{0}\equiv_{M}g_{1}$.
Hence $g_1^{-1} g p\vdash g_1^{-1} g_0 \phi(x) \Delta \phi(x)$ and $g_1^{-1} g_0 \in G^{00}$ (by Fact \ref{fac: BasicDefAm}(2)). Then $(g_1^{-1} g_0) \phi(x) \triangle \phi(x)$ is not $f$-generic by Lemma \ref{NonGDivFormulaG00Inv}, and so $p$ is not $f$-generic --- a contradiction.


Conversely, assume that some formula $\psi(x)$ implied by $p(x)$ is not f-generic. Let $N\supseteq M$ contain the parameters of $\psi$. Then there is some $(h_i)_{i<\omega}$ indiscernible over $N$ such that $\{h_i\psi(x)\}_{i<\omega}$ is $k$-inconsistent. Then $h_0\psi(x)\in h_0p$, but $h_i\psi(x)\notin h_0(p)$ for some $i<\omega$. So $h_0p$ is not $N$-invariant, and thus also not $M$-invariant.
\end{proof}
%
%

\begin{defi} \label{def: mu_p} Let $G = G(\M)$ be definably amenable, and let $p\in S_G(\monster)$ be f-generic. Keeping in mind that $p$ (as well as all its translates) is $G^{00}$-invariant (by Proposition \ref{prop:f-gen iff G00-inv}), we define a measure $\mu_p$ on $G$ by: 
$$\mu_p(\phi(x))=h_0(\{\bar{g}\in G/G^{00} : \phi(x)\in g\cdot p\}),$$ 
where $h_0$ is the normalized Haar measure on the compact group $G/G^{00}$ and $\bar g = g/G^{00}$. 
\end{defi}

We have to check that this definition makes sense, that is that the set we take the measure of is indeed measurable. Let $M$ be a small model over which $\phi(x)$ is defined. Let $p_M$ be the restriction of $p$ to formulas from $\mathcal{F}_M$ (as defined above). By Proposition \ref{prop: f-gen iff M-inv, locally}, $p_M$ is $M$-invariant. It follows that $p_M$ extends to some complete $M$-invariant type (by Remark \ref{rem: non-dividing extends to invariant}). Then we can use Borel-definability of invariant types (applied to the family of all translates of $\phi(x)$) exactly as in \cite[Proposition 5.6]{NIP2} to conclude.

\begin{rem}\label{rem: InvarianceOfMup}
\begin{enumerate}
\item The measure $\mu_p$ that we just constructed is clearly $G$-invariant and $G^{00}$-strongly invariant (that is, $\mu_p(\phi(x)\triangle g\cdot \phi(x))=0$ for $g\in G^{00}$). Besides $\mu_p = \mu_{gp}$ for any $g$, $p$.
\item We have $S(\mu_p) \subseteq \overline{G\cdot p}$. Indeed, if $q \in S(\mu_p)$ and $\phi(x) \in q$ arbitrary, then $\mu_p(\phi(x)) >0$, which by the definition of $\mu_p$ implies that $g \cdot p \vdash \phi(x)$ for some $g \in G$.
\end{enumerate}
\end{rem}

\begin{problem}\label{prob: f-gen ideal vs amenability}\footnote{We have claimed an affirmative answer in an earlier version of this article, however a mistake in our argument was pointed out by the referees.}
Let $G = G(\M)$ be an NIP group. Are the following two properties equivalent? 
\begin{enumerate}
\item $G$ is definably amenable.
\item $G$ admits a global $f$-generic type (equivalently, the family of all non-$f$-generic subsets of $G$ is an ideal).
\end{enumerate}

\end{problem}

\subsubsection{Approximation lemmas} \label{sec: approximation lemmas}
Throughout this section, $G=G(\M)$ is a definably amenable NIP group.
Given a $G^{00}$-invariant type $p(x) \in S_G(\M)$ and a formula $\phi(x) \in L_G(\M)$, let $A_{\phi,p} :=\{\bar g\in G/G^{00} : \phi(x)\in \bar g\cdot p\}$.

Note that $A_{g\cdot \phi,p} = \bar g \cdot A_{\phi,p}$ and $A_{\phi,g\cdot p}= A_{\phi,p}\cdot \bar g^{-1}$, where $\bar g$ is the image of $g$ in $G/G^{00}$.

\begin{lem}\label{lem: A_phi is VC}
For a fixed formula $\phi(x,y)$, let $\mathbf A_\phi \subseteq \mathfrak P(G/G^{00})$ be the family of all $A_{\phi(x,b),p}$ where $b$ varies over $\monster$ and $p$ varies over all f-generic types on $G$. Then $\mathbf A_\phi$ has finite VC-dimension.
\end{lem}
\begin{proof}
Let $\bar g_0, \ldots,  \bar g_{n-1}$ be shattered by $\mathbf A_\phi$. Then for any $I \subseteq n$ there is some $A_{\phi(x,b_I),p_I}$ which cuts out that subset. Take representatives $g_0, \ldots, g_{n-1} \in G$ of the $\bar g_i$'s. Let $a_I \models p_I |_{g_0, \ldots, g_{n-1} b_I}$, then we have $\phi(g_i a_I, b_I)$ if and only if $i \in I$. Hence the VC-dimension of $\mathbf A_\phi$ is at most that of $\psi(u; x,y) = \phi (u x,y)$, so finite by NIP.
\end{proof}

Replacing the formula $\phi(x;y)$ by $\phi'(x;y,u) :=\phi(u^{-1}\cdot x;y)$, we may  assume that any translate of an instance of $\phi$ is again an instance of $\phi$. Note also that then for any parameters $a,b$ we have $$\bar{g}_1 A_{\phi'(x;a,b), p} \bar{g}_2 = A_{g_1 \phi'(x;a,b), g_2^{-1} p} = A_{ \phi'(x;a',b'), g_2^{-1} p}$$ for some $a',b'$. Using this and applying Lemma \ref{lem: A_phi is VC} to $\phi'(x; y, u)$, we get the following corollary.

\begin{cor} \label{cor: g A_phi is VC}
For any $\phi(x,y)\in L_G(\M)$, the family  $$\mathcal{F}_\phi = \{\bar g_1 \cdot A_{\phi(x,b),p} \cdot \bar g_2 : \bar g_1, \bar g_2\in G/G^{00}, b \in \M, p \in S_G(\M) \textrm{ } f \textrm{-generic}\}$$ has finite VC-dimension. 
\end{cor}

We would now like to apply the VC-theorem to $\mathcal{F}_\phi$. This requires verifying an additional technical hypothesis (assumptions (2) and (3) in Fact \ref{VC-theorem}), which we are only able to show for certain (sufficiently representative) subfamilies of $\mathcal{F}_\phi$.

Fix $\phi(x) \in L_G(\M)$ and let $S$ be a set of global $f$-generic types. Let 
$$\mathcal{F}_{\phi, S} := \left\{ \bar g_1 \cdot A_{\phi(x), p} \cdot \bar g_2 : \bar g_1,\bar g_2 \in G/G^{00}, p \in S  \right\}.$$

\begin{lem} \label{lem_measurability}
If $S$ is countable and $L$ is countable, then  $\mathcal{F}_{\phi, S}$ satisfies all of the assumptions of Fact \ref{VC-theorem} with respect to the measure $h_0$.
\end{lem}

\begin{proof}
First of all, the family of sets $\mathcal{F}_{\phi, S}$ has finite VC-dimension by Corollary \ref{cor: g A_phi is VC} and the obvious inclusion $\mathcal{F}_{\phi, S} \subseteq \mathcal{F}_\phi$.

Next, (1) is satisfied by the assumption that $S$ consists of $f$-generic types and an argument as in the discussion after Definition \ref{def: mu_p} (using countability of the language).

For a set $S'$ of global $f$-generic types, let 
$$f_{S', n}(x_0, \ldots, x_{n-1}) := \sup_{Y \in \mathcal{F}_{\phi,S'}} \{ |\Av(x_0, \ldots, x_{n-1};Y) - h_0(Y)| \} \textrm{,}$$
$$g_{S',n}(x_0, \ldots, x_{n-1}, y_0, \ldots, y_{n-1}) := \sup_{Y \in \mathcal{F}_{\phi,S'}} \{ |\Av(x_0, \ldots, x_{n-1};Y) - $$
$$ - \Av(y_0, \ldots, y_{n-1};Y)| \} \textrm{.}$$

For (2) and (3) we need to show that $f_{S,n}$ and $g_{S,n}$ are measurable for all $n < \omega$. Note that $f_{S,n} = \sup_{p \in S} f_{\{ p\}, n}$ and $g_{S,n} = \sup_{p \in S} g_{\{ p \}, n}$. Since $S$ is countable, it is enough to show that for a fixed $f$-generic type $p$ the functions $f_n := f_{\{p\}, n}$ and $g_n := g_{\{p\}, n}$ are measurable. 

Let $A = A_{\phi,p}$. By $G/G^{00}$-invariance of $h_0$ both on the left and on the right, we have:
$$f_n(x_0,\ldots,x_{n-1})=\max_{\bar g_1, \bar g_2 \in G/G^{00}} |\Av(x_0,\ldots,x_{n-1};\bar g_1 \cdot A \cdot \bar g_2) -h_0(A)|$$ and $$g_n(x_0,\ldots,x_{n-1},y_0,\ldots,y_{n-1})=\max_{\bar g_1, \bar g_2 \in G/G^{00}} |\Av(x_0,\ldots,x_{n-1};\bar g_1 \cdot A \cdot \bar g_2)-$$
$$-\Av(y_0,\ldots,y_{n-1};\bar g_1\cdot A \cdot \bar g_2)|\textrm{.}$$

Then it is enough to show that for a fixed $I\subseteq n$, the set 
$$A_I=\{(x_0,\ldots,x_{n-1}) \in (G/G^{00})^n :  \textrm{ for some } \bar g_1, \bar g_2 \in G/G^{00},$$
$$ x_i\in \bar g_1 \cdot A \cdot \bar g_2 \iff i\in I\}$$
 is measurable. But we can write $A_I$ as the projection of $A'_I\subseteq (G/G^{00})^{n+2}$ where $A'_I$ is the intersection of $\{(\bar g_1, \bar g_2,x_0,\ldots,x_{n-1}) : \bar g_1^{-1}x_i \bar g_2^{-1} \in A\}$ for $i\in I$ and $\{(\bar g_1, \bar g_2 ,x_0,\ldots,x_{n-1}) : \bar g_1^{-1}x_i \bar g_2^{-1} \notin A\}$ for $i\notin I$. As group multiplication is continuous and $A$ is Borel, those sets are Borel as well. Hence $A_I$ is analytic. Now $G/G^{00}$ is a Polish space (as $L$ is countable, by Remark \ref{rem: G/G^{00} is Polish}) and analytic subsets of Polish spaces are universally measurable (see e.g. \cite[Theorem 29(7)]{Kechris1995}). In particular they are measurable with respect to the complete Haar measure $h_0$.
\end{proof}


The next lemma will allow us to reduce to a countable sublanguage.
\begin{lem}\label{lem: reducing measures to countable sublanguage}
Let $L_0$ be a sublanguage of $L$, $T_0$ the $L_0$-reduct of $T$, $G$ an $L_0$-definable group definably amenable (in the sense of $T$) and $\phi(x)$ a formula from $L_0(\M)$. Let $p \in S_G(\M)$ be a global $L$-type which is $f$-generic, and let $p_0 = p|_{L_0}$.
\begin{enumerate}
\item In the sense of $T_0$, the group $G$ is definably amenable NIP and $p_0$ is an $f$-generic type.
\item Let $G^{00}_{L_0}$ be the connected component computed in $T_0$, and let $\mu_{p_0}$ ($\mu_p$) be the $G$-invariant measure on $L_0$-definable (resp., $L$-definable) subsets of $G$ given by Definition \ref{def: mu_p} in $T_0$ (resp., in $T$). Then $\mu_p(\phi(x)) = \mu_{p_0}(\phi(x))$.
\end{enumerate}
\end{lem}
\begin{proof}
(1) The first assertion is clear. Similarly, it is easy to see that if $\psi(x) \in L_0$ is $G$-dividing in $T_0$, then it is $G$-dividing in $T$ (by extracting an $L$-indiscernible sequence from an $L_0$ indiscernible sequence). Then $p_0$ is $f$-generic by Fact \ref{prop_Gdivide} applied in $T_0$.

(2) Let $A = \{ \bar{g} \in G/G^{00} : g\cdot p \vdash \phi(x) \}$ and $A_0 = \{ \bar{g} \in G/G^{00}_{L_0} : g\cdot p_0 \vdash \phi(x) \}$, then by definition $\mu_p(\phi(x)) = h_0(A)$ and $\mu_{p_0}(\phi(x)) = h'_0(A_0)$, where $h_0$ is the Haar measure on $G/G^{00}$ and $h'_0$ is the Haar measure on $G/G^{00}_{L_0}$. The map $f : G/G^{00} \to G/G^{00}_{L_0}, g/G^{00} \mapsto g/G^{00}_{L_0}$ is a surjective group homomorphism, and it is continuous with respect to the logic topology. Note that for any $g\in G$ we have $g\cdot p_0 \vdash \phi(x) \iff g\cdot p \vdash \phi(x)$, so $A = f^{-1}(A_0)$. Let $h^*_0 = f_*(h_0)$ be the push-forward measure, it is an invariant measure on $G/G^{00}_{L_0}$. But by the uniqueness of the Haar measure, it follows that $h^*_0 = h'_0$, and so $h_0(A) = h^*_0(A_0) = h'_0(A_0)$, i.e., $\mu_p(\phi(x)) = \mu_{p_0}(\phi(x))$ as wanted.
\end{proof}


\begin{prop} \label{prop: countable approximation}
For any $\phi(x) \in L_G(\M)$, $\varepsilon>0$ and a \emph{countable} set of f-generic types $S \subseteq S_G(\M)$  there are some $g_0, \ldots, g_{n-1} \in G$ such that: for any $g, g' \in G$ and $p \in S$ we have $\mu_{gp}(g' \phi(x)) \approx^{\varepsilon} \Av(g_j g' \phi(x) \in gp)$.
\end{prop}
\begin{proof}
First assume that the language $L$ is countable.
Using Lemma \ref{lem_measurability}, we can apply the VC-theorem (Fact \ref{VC-theorem}) to the family $\mathcal{F}' = \mathcal{F}_{\phi, S}$ and find some $\overline{g}_0, \ldots, \overline{g}_{n-1} \in G/G^{00}$ such that for any $Y \in \mathcal{F}'$ we have $\Av(\overline{g}_0, \ldots, \overline{g}_{n-1}; Y) \approx^{\varepsilon} h_0(Y)$. Let $g_i \in G$ be some representative of $\overline{g}_i$, for $i<n$. Let $g, g' \in G$ and $p \in S$ be arbitrary. Recall that $\mu_{gp}(g' \phi(x)) = h_0(A_{g' \phi,gp})$ and that $A_{g'\phi, gp} = \overline{g'} A_{\phi, p}\overline{g}^{-1}$, where $\overline{g} = g/G^{00}, \overline{g}' = g'/G^{00}$. Then $A_{g'\phi, g p} \in \mathcal{F}'$ and we have $\mu_{gp}(g' \phi(x)) \approx^{\varepsilon} \Av (\overline{g}_0, \ldots, \overline{g}_{n-1}; A_{g' \phi, g p}) = \Av(g^{-1}_0 g' \phi(x), \ldots, g^{-1}_{n-1} g' \phi(x); g p)$.

Now let $L$ be an arbitrary language, and let $L_0$ be an arbitrary countable sublanguage such that $\phi(x) \in L_0$ and $G$ is $L_0$-definable, let $T_0$ be the corresponding reduct. Let $S_0 = \{ p|_{L_0} : p \in S \}$, by Lemma \ref{lem: reducing measures to countable sublanguage} it is a countable set of $f$-generic types in the sense of $T_0$. Applying the countable case with respect to $S_0$ inside $T_0$, we find some $g_0, \ldots, g_{n-1} \in G$ such that for any $g,g' \in G$ and $p_0 \in S_0$ we have $\mu_{gp_0}(g' \phi(x)) \approx^{\varepsilon} \Av(g_j g' \phi(x) \in g p_0)$. Let $p \in S$ be arbitrary, and take $p_0 = p|_{L_0}$. On the one hand, the right hand side is equal to $\Av(g_j g' \phi(x) \in gp)$. On the other hand, as $g'\phi(x) \in L_0(\M)$ and $gp_0 = gp|_{L_0}$ is $f$-generic,  by Lemma \ref{lem: reducing measures to countable sublanguage} the left hand side is equal to $\mu_{gp}(g' \phi(x))$, as wanted.
\end{proof}

\begin{prop}\label{OrbitPreservesMuP}
Let $p$ be an f-generic type, and assume that $q \in \overline{G\cdot p}$. Then $q$ is f-generic and $\mu_p = \mu_q$.
\end{prop}
\begin{proof}
First of all, $q$ is f-generic because the orbit of $p$ consists of f-generic types and the set of f-generic types is closed.

Take a formula $\phi(x)\in L_G(\monster)$ and $\varepsilon > 0$, and let $g_0, \ldots, g_{n-1}$ be as given by Proposition \ref{prop: countable approximation} for $S = \{p,q\}$. 
Then we have $\mu_q(\phi(x)) \approx^{\varepsilon} \Av(g_i \phi(x) ; q)$.
As $q \in \overline{G\cdot p}$, there is some $g \in G$ such that for each $i<n$ we have $g_i \phi(x) \in q \iff g_i \phi(x) \in g p$. But we also have $\mu_{g p}(\phi(x)) \approx^{\varepsilon} \Av(g_i \phi(x); gp)$, which together with $\mu_{gp} = \mu_{p}$ implies $\mu_{p} (\phi(x)) \approx^{2 \varepsilon} \mu_q(\phi(x))$. As $\phi(x)$ and $\varepsilon$ were arbitrary, we conclude.
\end{proof}

\begin{prop}\label{prop_locallygeneric}
Let $p$ be an f-generic type. Then for any definable set $\phi(x)$, if $\mu_p(\phi(x)) > 0$, then there is a finite union of translates of $\phi(x)$ which covers the support $S(\mu_p)$ (so in particular has $\mu_p$-measure $1$).
\end{prop}
\begin{proof}
As $S(\mu_p) \subseteq \overline{G\cdot p}$ (Remark \ref{rem: InvarianceOfMup}), any type $q$ weakly random for $\mu_p$ is f-generic and satisfies $\mu_q = \mu_p$ by Proposition \ref{OrbitPreservesMuP}. Hence $\mu_q(\phi(x))>0$, so some translate of $\phi(x)$ must be in $q$. It follows that the closed compact set $S(\mu_p)$ can be covered by translates of $\phi$, so by finitely many of them.
\end{proof}


\begin{lem}\label{lem_ApproxInvMeasByMuP}
Let $\mu$ be $G$-invariant. Then for any $\varepsilon >0$ and $\phi(x,y)$, there are some f-generic $p_0, \ldots, p_{n-1} \in S(\mu)$ such that 
$$\mu(\phi(x,b)) \approx^{\varepsilon} \frac 1 n \sum_{i<n} \mu_{p_i}(\phi(x,b))$$ for any $b \in \M$.
\end{lem}

\begin{proof}

As before, we may assume that every translate of an instance of $\phi(x,y)$ is an instance of $\phi(x,y)$. Fix $\varepsilon > 0$.

By Fact \ref{fac: measure is average of types} there are some $p_0, \ldots, p_{n-1} \in S(\mu)$ such that $\mu(\phi(x,b)) \approx^{\varepsilon} \Av(\phi(x,b) \in p_i)$ for all $b \in \M$.  It follows by $G$-invariance of $\mu$ and the assumption on $\phi$ that for any $g \in G$ and $b \in \M$, $\Av(g \phi(x,b) \in p_i) \approx^{\varepsilon} \mu(\phi(x,b))$.

By Proposition \ref{lem_PosGInvMeasImpliesFGen}, all of the $p_i$'s are f-generic. By Proposition \ref{prop: countable approximation} with $S=\{ p_0, \ldots, p_{n-1}\}$, for every $b \in \M$ there are some $g_0, \ldots, g_{m-1} \in G$ such that for any $i<n$, $\mu_{p_i}( \phi(x,b)) \approx^{\varepsilon} \Av(g_j  \phi(x,b) \in p_i)$.

So let $b \in \M$ be arbitrary, and choose the corresponding $g_0, \ldots, g_{m-1}$ for it. By the previous remarks we have $$\frac{1}{n}\sum_{i<n}\mu_{p_i}(\phi(x,b)) \approx^{\varepsilon} \frac{1}{n} \sum_{i<n} \Av(g_j \phi(x,b) \in p_i) = $$ $$=\frac{1}{n} \sum_{i<n} \left( \frac{1}{m} \sum_{j<m} "g_j \phi(x,b) \in p_i" \right) =  \frac{1}{m} \sum_{j<m} \left( \frac{1}{n} \sum_{i<n} "g_j \phi(x,b) \in p_i" \right) = $$
$$= \frac{1}{m} \sum_{j<m} \Av(g_j \phi(x,b) \in p_i) \approx^{\varepsilon} \frac{1}{m} \sum_{j<m} \mu(\phi(x,b)) = \mu(\phi(x,b)) \textrm{.}$$ Thus $\mu(\phi(x,b)) \approx^{2 \varepsilon} \frac 1 n \sum_{i<n} \mu_{p_i}(\phi(x,b))$.
\end{proof}

\begin{cor} \label{cor: mu_p approx via support}
Let $\mu$ be a $G$-invariant measure and assume that $S(\mu) \subseteq \overline{G\cdot p}$ for some $f$-generic $p$. Then $\mu = \mu_p$.
\end{cor}
\begin{proof}
Let $\phi(x) \in L_G(\M)$ and $\varepsilon > 0$ be arbitrary. By Lemma \ref{lem_ApproxInvMeasByMuP} we can find some $f$-generic $p_0, \ldots, p_{n-1} \in S(\mu)$ such that $\mu(\phi(x)) \approx^{\varepsilon} \Av(\mu_{p_i}(\phi(x)) : i<n)$. But as $p_i \in S(\mu) \subseteq \overline{G\cdot p}$, it follows by Proposition \ref{OrbitPreservesMuP} that $\mu_{p_i} = \mu_p$ for all $i<n$, so $\mu(\phi(x)) \approx^{\varepsilon} \mu_{p}(\phi(x))$.
\end{proof}

\subsection{Weak genericity and almost periodic types}

Now we return to the notions of genericity for definable subsets of definable groups and add to the picture another one motivated by topological dynamics, due to Newelski.

We will be using the standard terminology from topological dynamics: Given a group $G$, a $G$-flow is a compact space $X$ equipped with an action of $G$ such that every $x\mapsto g\cdot x$, $g\in G$ is a homeomorphism of $X$. We will usually write a $G$-flow $X$ as a pair $(G,X)$. A set $Y \subseteq X$ is said to be a \emph{subflow} if $Y$ is closed and $G$-invariant. The flows relevant to us are $(S_G(\M),G(\M))$ and $(S_G(M),G(M))$ for a small model $M$.

\begin{defi}[\cite{New4, poizat1987groupes}] \label{def: weak generics, etc}
\begin{enumerate} 
\item A formula $\phi(x) \in L_G(\M)$ is (left-)\emph{generic} if there are some finitely many $g_0, \ldots, g_{n-1} \in G$ such that $G = \bigcup_{i<n} g_i \phi(x)$.
\item A formula $\phi(x) \in L_G(\M)$ is (left-)\emph{weakly generic} if there is formula $\psi(x)$ which is not generic, but such that $\phi(x) \lor \psi(x)$ is generic.
\item A (partial) type is (weakly) generic if it only contains (weakly) generic formulas.
\item A type $p \in S_G(\M)$ is called almost periodic if it belongs to a minimal flow in $(S_G(\M),G(\M))$ (i.e., a minimal $G$-invariant closed set), equivalently if for any $q \in \overline{G\cdot p}$ we have $\overline{G\cdot p} = \overline{G\cdot q}$.
\end{enumerate}
\end{defi}

\begin{fact} [\cite{New4}, Section 1] \label{BasicWeakGenAP}
The following hold, in an arbitrary theory:
\begin{enumerate}
\item The formula $\phi(x)$ is weakly generic if and only if for some finite $A \subseteq G$, $X \setminus (A\cdot \phi(x))$ is not generic.
\item The set of non weakly generic formulas forms a $G$-invariant ideal. In particular, there are always global weakly generic types by compactness.

\item The set of all weakly generic types is exactly the closure of the set of all almost periodic types in $S_G(\M)$.

\item Every generic type is weakly generic. Moreover, if there is a global generic type then every weakly generic type is generic, and the set of generic types is the unique minimal flow in $(S_G(\M),G(\M))$.

\item A type $p(x)$ is almost periodic if and only if for every $\phi(x) \in p$, the set $\overline{G\cdot p}$ is covered by finitely many left translates of $\phi(x)$.

\end{enumerate}
\end{fact}


%

We connect these definitions to the notions of genericity from the previous sections. As before, we always assume that $G = G(\M)$ is NIP.

\begin{prop} \label{prop_WGenImpliesFGen}
Let $G$ be definably amenable and let $\phi(x) \in L_G(M)$ be a weakly generic formula. Then it is f-generic.
\end{prop}
\begin{proof}
We adapt the argument from \cite[Lemma 1.8]{NewPetr}. As $\phi(x)$ is weakly generic, let $\psi(x)$ be non-generic and $A \subset G$ a finite set such that $A \cdot (\phi(x) \lor \psi(x)) = X$. We may assume that $A \subset M$ and that $\psi(x)$ is defined over $M$. Assume that $\phi(x)$ is not f-generic over $M$. The set of formulas which are not f-generic is $G$-invariant, and moreover it is an ideal by Corollary \ref{prop_ideal}. Thus $A \cdot \phi(x)$ is not f-generic, which implies that there is some $g \in G$ such that $g \cdot A \cdot \phi(x)$ divides over $M$. That is, there is an $M$-indiscernible sequence $(g_i)_{i < k}$ such that $\bigcap_{i < k} g_i \cdot A \cdot \phi(x) = \emptyset$.

As $A \cdot \phi(x) \cup A \cdot \psi(x) = G$, we also have $g_i \cdot A \cdot \phi(x) \cup g_i \cdot A \cdot \psi(x) = G$ for every $i<k$. Thus $G \setminus \bigcup_{i<k} g_i \cdot A \cdot \psi(x) \subseteq \bigcap_{i<k} g_i \cdot A \cdot \phi(x) = \emptyset$. But this means that $\psi(x)$ is generic, a contradiction.
\end{proof}

\begin{prop} \label{prop: ap implies f-generic}
Assume that $G$ is definably amenable.
\begin{enumerate}
\item If $p$ is almost periodic then it is $f$-generic and $\overline{G \cdot p} = S(\mu_p)$.
\item Minimal flows in $S_G(\M)$ are exactly the sets of the form $S(\mu_p)$ for some $f$-generic $p$.
\item If $p,q$ are almost periodic and $\mu_p = \mu_q$ then $\overline{G\cdot p} = \overline{G\cdot q}$.
\end{enumerate}
\end{prop}
\begin{proof}

(1) An almost periodic type $p$ contains only weakly generic formulas and hence is $f$-generic by Proposition \ref{prop_WGenImpliesFGen}. As $S(\mu_p) \subseteq \overline{G\cdot p}$ (see Remark \ref{rem: InvarianceOfMup}), it follows by minimality that $S(\mu_p) = \overline{G\cdot p}$.

(2) For an $f$-generic $p$, the set $S(\mu_p)$ is a subflow by $G$-invariance of $\mu_p$. If $q \in S(\mu_p)$ and $\phi(x) \in q$, then $\mu_p(\phi(x)) > 0$ and by Proposition \ref{prop_locallygeneric} there are finitely many translates of $\phi(x)$ which cover $S(\mu_p)$, so in particular they cover $\overline{G\cdot q} \subseteq S(\mu_p)$. Thus $q$ is almost periodic (by the usual characterization of almost periodic types from Fact \ref{BasicWeakGenAP}(5)).

(3) is clear.
\end{proof}

In particular, for any $f$-generic type $p$ there is some almost periodic type $q$ with $\mu_p = \mu_q$. However, the following question remains open \footnote{While this paper was under review, a negative answer was obtained in \cite{pillay2016minimal}.}.
\begin{problem}\label{question}
Is every $f$-generic type almost periodic? Equivalently, does $p \in S(\mu_p)$ always hold? 
\end{problem}

Now towards the converse. 

\begin{prop} \label{prop: non-G-div extends to AP}
Let $G$ be definably amenable. Assume that $\phi\left(x\right)$ does not $G$-divide. Then there are some global almost periodic types $p_0, \ldots, p_{n-1} \in S_G(\M)$ such that for any $g \in G$ there is some $i<n$ such that $g \phi(x) \in p_i$ holds.

\end{prop}
\begin{proof}

Let $k \in \omega$ be as given by Fact \ref{fac: p,q-theorem} for the VC-family  $\mathcal{F} = \{ g\phi\left(x\right):g\in G \} $. We claim that $\mathcal{F}$ satisfies the $\left(p,k\right)$-property for some $p < \omega$. If not, then by compactness we can find an infinite indiscernible sequence $(g_i)_{i < \omega}$ in $G$ such that $\{ g_i \phi(x) : i < \omega\}$ is $k$-inconsistent, and so $G$-divides.

By Fact \ref{fac: p,q-theorem} and compactness it follows that there are some $p_{0}, \ldots, p_{N-1} \in S_{G}\left(\M\right)$
which satisfy:
\begin{enumerate}
\item [{$\left(*\right)$}] for every $g\in G$, for some $i<N$, we have $g\phi\left(x\right)\in p_{i}$. 
\end{enumerate}
Now consider the action of $G$ on $\left(S_{G}\left(\M\right)\right)^{N}$
with the product topology, and let $$F=\overline{\left\{ g\cdot\left(p_{0},\ldots,p_{N-1}\right):g\in G\right\} }.$$
It is a subflow, and besides every $\left(q_{0},\ldots,q_{N-1}\right)\in F$
satisfies $\left(*\right)$ (it is clear for translates of $\left(p_{0},\ldots,p_{N-1}\right)$;
if for some $g\in G$ we have $\bigwedge_{i<N}\neg g\cdot \phi\left(x_{i}\right)\in q_{i}$,
then since $\bigwedge_{i<N}\neg g\cdot \phi\left(x_{i}\right)$ is an open
subset of $\left(S_{G}\left(\M\right)\right)^{N}$
with respect to the product topology containing $\left(q_{0},\ldots,q_{N-1}\right)$, it follows that $h\cdot\left(p_{0},\ldots,p_{N-1}\right)$
belongs to it for some $h\in G$, which is impossible). Let $F'$
be a minimal subflow of $F$, and notice that the projection of $F'$
on any coordinate is a minimal subflow of $\left(G,S_{G}\left(\M\right)\right)$.
Thus, taking $\left(q_{0},\ldots,q_{N-1}\right)\in F'$, it follows
that $q_{i}$ is almost periodic for every $i<N$, and every translate
of $\phi\left(x\right)$ belongs to one of the $q_i, i < N$.
\end{proof}

\begin{cor} \label{cor_FGenImpliesWGen} Let $G$ be definably amenable.
If $\phi(x)$ is f-generic, then $\mu_q(\phi(x))>0$ for some global f-generic type $q$.
\end{cor}
\begin{proof}
Let $p_0, \ldots, p_{n-1}$ be some global almost periodic types given by Proposition \ref{prop: non-G-div extends to AP}, they are also $f$-generic by Proposition \ref{prop: ap implies f-generic}. Let $Y_i = \{ \bar g \in G/G^{00} : g\phi(x) \in p_i \}$. As $\bigcup_{i<n} Y_i = G/G^{00}$ and each of $Y_i$'s is measurable, it follows that $h_0(Y_i) \geq \frac{1}{n}$ for some $i<n$. But then $\mu_{p_i}(\phi(x)) \geq \frac{1}{n}$.
\end{proof}

Summarizing, we have demonstrated that all notions of genericity that we have considered coincide in definable amenable NIP groups.

\begin{thm}\label{thm:AllGenericNotionsCoincide} Let $G$ be definably amenable, NIP.
Let $\phi(x)$ be a definable subset of $G$. Then the following are equivalent:

\begin{enumerate}
\item $\phi(x)$ is f-generic;

\item $\phi(x)$ is not $G$-dividing;

\item $\phi(x)$ is weakly-generic;

\item $\mu(\phi(x))>0$ for some $G$-invariant measure $\mu$;

\item $\mu_p(\phi(x))>0$ for some global $f$-generic type $p$.
\end{enumerate}
\end{thm}

\begin{proof}
(1) and (2) are equivalent by Proposition \ref{prop_Gdivide}, (1) implies (3) by Proposition \ref{prop: non-G-div extends to AP} and (3) implies (1) by Proposition \ref{prop_WGenImpliesFGen}. Finally, (1) implies (5) by Corollary \ref{cor_FGenImpliesWGen}, (5) implies (4) is obvious and (4) implies (1) by Lemma \ref{lem_PosGInvMeasImpliesFGen}.
\end{proof}


\subsection{Unique ergodicity} \label{sec: unique ergodicity}

We now characterize the case when $G$ admits a unique $G$-invariant measure. Following standard terminology in topological dynamics, we call such a $G$ uniquely ergodic (indeed, it will follow from the next section that this condition is equivalent to $S_G(\monster)$ having a unique regular ergodic measure).

Recall that a $G$-invariant measure $\mu$ is called \emph{generic} if for any definable set $\phi(x)$, $\mu(\phi(x))>0$ implies that $\phi(x)$ is generic. It follows that any $p \in S(\mu)$ is generic.

\begin{thm}\label{thm: unique ergodic}
A definably amenable NIP group $G$ is uniquely ergodic if and only if it admits a generic type (in which case it has a unique minimal flow --- the support of the unique measure).
\end{thm}
\begin{proof}
If $G$ admits a generic type $p$, then for any type $q$, $p$ belongs to the closure $\overline{G\cdot q}$ (if $\phi(x) \in p$ then $X = \bigcup_{i<n} g_i\cdot \phi(x)$ for some $g_i \in G$, so $\phi(x) \in g^{-1}_i q$ for some $i<n$). In particular, for an arbitrary $f$-generic type $q$ we have $\mu_q = \mu_p$ (by Proposition \ref{OrbitPreservesMuP}). By Lemma \ref{lem_ApproxInvMeasByMuP}, this implies that any invariant measure $\mu$ is equal to $\mu_p$, hence there is a unique invariant measure.

Conversely, assume that $G$ admits a unique $G$-invariant measure $\mu$. We claim that $\mu$ is generic. Assume not, and let $\phi(x)$ be a definable set of positive $\mu$-measure and assume that $\phi(x)$ is not generic. Then for any $g_1,\ldots,g_n \in G$, the union $\bigvee_{i<n} g_i\cdot \phi(x)$ is not generic. Hence its complement is weakly generic. By Theorem \ref{thm:AllGenericNotionsCoincide} we conclude that the partial type $\{\neg g\cdot \phi(x) : g\in G(\monster)\}$ is $f$-generic and hence extends to a complete $f$-generic type $p$. The measure $\mu_p$ associated to $p$ gives $\phi(x)$ measure 0, so $\mu_p \neq \mu$, which contradicts unique ergodicity.
\end{proof}

\begin{rem}
In particular, in a uniquely ergodic group every $f$-generic type is almost periodic and generic.
\end{rem}

Recall from \cite{NIP2} that an NIP group $G$ is \emph{fsg} if it admits a global type $p$ such that for some small model $M$, all translates of $p$ are finitely satisfiable over $M$. It is proved that an fsg group admits a unique invariant measure and that this measure is generic. So the previous proposition was known in this special case. We now give an example (pointed out to us by Hrushovski) of a uniquely ergodic group which is not fsg.

\begin{rem}
Let $K_v$ be a model of ACVF and consider $G=(K_v,+)$ the additive group.  By C-minimality, the partial type $p$ concentrating on the complement of all balls is a complete type and is $G$-invariant. There can be no other $G$-invariant measure since any non-trivial ball in $(K_v,+)$ $G$-divides, hence cannot have positive measure for any $G$-invariant measure. Finally, the group $G$ is not fsg since $p$ is not finitely satisfiable.
\end{rem}

\section{Regular ergodic measures} \label{sec: ergodicity}

In this section, we are going to characterize regular ergodic measures on $S_{G}(\M)$ for a definably amenable NIP group $G=G(\M)$, but first we recall some general notions and facts from functional analysis and ergodic theory (see e.g. \cite{walters2000introduction}). As we are going to deal with more delicate measure-theoretic issues here, we will be specific about our measures being regular or not. The reader should keep in mind that all the results in the previous sections only apply to regular measures on $S_G(\M)$.

The set of all regular (Borel, probability) measures on $S_G(\M)$ can be naturally viewed as a subset of $C^*(S_G(\M))$, the dual space of the topological vector space of continuous functions on $S_G(\M)$, with the weak$^*$ topology of pointwise convergence (i.e., $\mu_i \to \mu$ if $\int f d\mu_i \to \int f d \mu$ for all $f \in C(S_G(\M))$). It is easy to check that this topology coincides with the logic topology on the space of measures (Remark \ref{rem: TopOnMeas}). This space carries a natural structure of a real topological vector space containing a compact convex set of $G$-invariant measures.

We will need the following version of a ``converse'' to the Krein-Milman Theorem (see e.g. \cite[Theorem 1]{Jerison1954}. We refer to e.g. \cite[Chapter 8]{simon2011convexity} for a discussion of convexity in topological spaces).
\begin{fact} \label{fac: KreinMilman}
Let $E$ be a real, locally convex, Hausdorff topological vector space. Let $C$ be a compact convex subset of $E$, and let $S$ be a subset of $C$. Then the following are equivalent:
\begin{enumerate}
\item $C = \overline{\conv}S$, the closed convex hull of $S$.
\item The closure of $S$ includes all extreme points of $C$.
\end{enumerate}
\end{fact}

Now we recall the definition of an ergodic measure.

\begin{fact} [{\cite[Proposition 12.4]{phelps2001lectures}}] \label{def: ergodic} Let $G$ be a group acting on a topological space $X$ with $x \mapsto gx$ a Borel map for each $g \in G$, and let $\mu$ be a $G$-invariant Borel probability measure on $X$. Then the following are equivalent:
\begin{enumerate}
\item The measure $\mu$ is an extreme point of the convex set of $G$-invariant measures on $X$.
\item For every Borel set $Y$ such that $\mu(gY \triangle Y) = 0$ for all $g \in G$, we have that either $\mu(Y) = 0$ or $\mu(Y) = 1$.
\end{enumerate}
\end{fact}

A $G$-invariant measure is \emph{ergodic} if it satisfies any of the equivalent conditions above. Under many natural conditions on $G$ and $X$ the two notions above are equivalent to the following property of $\mu$: for every $G$-invariant Borel set $Y$, either $\mu(Y) = 0$ or $\mu(Y) = 1$. However this is not the case in general.



\begin{prop} \label{prop: p to mu_p is cont}
The map $p \mapsto \mu_p$ from the (closed) set of global $f$-generic types to the (closed) set of global $G$-invariant measures on $S_G(\M)$ is continuous.
\end{prop}
\begin{proof}
Fix $\phi\left(x\right)\in L_G\left(\M\right)$ and $r\in\left[0,1\right]$,
and let $Y$ be the set of all global $f$-generic $p \in S_G(\M)$ with $\mu_{p}\left(\phi\left(x\right)\right)\geq r$.
It is enough to show that $Y$ is closed. Let $q$ belong to the closure of $Y$, in particular $q$ is $f$-generic. Let $L_0$ be some countable language such that $G$ is $L_0$-definable and $\phi(x) \in L_0(\M)$, and let $T_0 = T|_{L_0}$.

Now let $M$ be some countable model of $T_0$ over which $\phi(x)$ is defined, and let $\psi(x,y) = \phi(y^{-1}x)$. Let $q'(x) = q|_{\psi}$, i.e., the restriction of $q$ to all formulas of the form $g \cdot \phi(x), \neg g \cdot \phi(x)$, $g \in G$, and let $Y' = \{ p|_{\psi} : p \in Y \}$. By Lemma \ref{lem: reducing measures to countable sublanguage}, $q'$ and all elements of $Y'$ are $f$-generic in the sense of $T_0$. By Lemma \ref{prop: f-gen iff M-inv, locally} applied in $T_0$ we know that $q'$ and all elements of $Y'$ are $M$-invariant. Working in $T_0$, let $\Inv_{\psi}(M)$ be the space of all global $\psi$-types invariant over $M$. It follows from the assumption that $q' \in \overline{Y'}$ (i.e., the closure of $Y'$ in the sense of the topology on $\Inv_\psi(M)$).

By Fact \ref{fac: seq compactness}  we know that $q'$ is a limit of a countable sequence $(p'_i : i < \omega)$ of types from $Y'$. Each of $p'_i$ is $f$-generic in $T_0$, so in $T$ as well (easy to verify using equivalence to $G$-dividing both in $T$ and $T_0$), and extends to some global $f$-generic $L$-type $p_i \in Y$ by Corollary  \ref{prop_ideal}. 

Now work in $T$, and let $\varepsilon > 0$ be arbitrary. By Proposition \ref{prop: countable approximation}, with $S = \{ q \} \cup \{ p_i : i < \omega \}$, there are some $g_0, \ldots, g_m \in G$ such that $\mu_{p_i}(\phi(x)) \approx^{\varepsilon} \Av(g_j \phi(x) \in p_i)$ for all $i < \omega$, as well as $\mu_q(\phi(x)) \approx^{\varepsilon} \Av(g_j \phi(x) \in q)$. As for any $g \in G$, $g \phi(x) \in p_i \iff g \phi(x) \in p'_i$, and the same for $q,q'$, it follows that for all $i < \omega$ large enough we have $\bigwedge_{j<m} (g_j \phi(x) \in q \iff g_j \phi(x) \in p_i)$. But this implies that for any $\varepsilon > 0$, $\mu_{q}\left(\phi\left(x\right)\right)\geq r-\varepsilon$, and so
$\mu_{q}\left(\phi\left(x\right)\right)\geq r$ and $q\in Y$.
\end{proof}

\begin{cor} \label{cor: implications of f continuous}
\begin{enumerate}
\item The set $\{ \mu_p : p \mbox{ is } f \mbox{-generic} \}$ is closed in the set of all $G$-invariant measures.
\item Given a $G$-invariant measure $\mu$, the set of f-generic types $p$ for which $\mu_p = \mu$ is a subflow.
\end{enumerate}
\end{cor}
\begin{proof}
This follows from Proposition \ref{prop: p to mu_p is cont}.
\end{proof}

\begin{thm} \label{thm: egodic iff mu_p}
Let $G$ be definably amenable. Then regular ergodic measures on $S_G(\M)$ are exactly the measures of the form $\mu_p$ for some f-generic $p \in S_G(\M)$.
\end{thm}
\begin{proof}
Fix a global $f$-generic type $p$, and assume that $\mu_p$ is not an extreme point. Then there is some $0 < t < 1$ and some $G$-invariant measures $\mu_1 \neq \mu_2$ such that $\mu_p = t \mu_1 + (1-t) \mu_2$. First, it is easy to verify using regularity of $\mu_p$ that both $\mu_1$ and $\mu_2$ are regular. 
Second, it follows that $S(\mu_1), S(\mu_2) \subseteq S(\mu_p) \subseteq \overline{Gp}$. By Corollary \ref{cor: mu_p approx via support}  which we may apply as $\mu_1, \mu_2$ are regular, it follows that $\mu_1 = \mu_p = \mu_2$, a contradiction.

%
%
%
Now for the converse, let $\mu$ be an arbitrary regular $G$-invariant measure which is an extreme point, and let $S = \{ \mu_p : p \in S_G(\M) \mbox{ is } f \mbox{-generic} \}$. Let $\overline{\conv}S$ be the closed convex hull of $S$. By Lemma \ref{lem_ApproxInvMeasByMuP}, $\mu$ is a limit of the averages of measures from $S$, so  $\mu \in \overline{\conv}S$ and it is still an extreme point of $\overline{\conv}S$. Then we actually have $\mu \in \overline{S}$ (by Fact \ref{fac: KreinMilman}, as (1) is automatically satisfied for $C = \overline{\conv}S$, then (2) holds as well). But $\overline{S} = S$ by Corollary \ref{cor: implications of f continuous}(1).
\end{proof}

\begin{cor}
The set of all regular ergodic measures in $S_G(\M)$ is closed.
\end{cor}

Let $\FGen$ denote the closed $G$-invariant set of all $f$-generic types in $S_G(\M)$. By Proposition \ref{prop:f-gen iff G00-inv} we have a well-defined action of $G/G^{00}$ on $\FGen$ (not necessarily continuous, or even measurable). If $\nu$ is an arbitrary regular $G$-invariant measure, then $S(\nu) \subseteq \FGen$ by Proposition \ref{lem_PosGInvMeasImpliesFGen}, and we can naturally view $\nu$ as a $G/G^{00}$-invariant measure on Borel subsets of $\FGen$.

\begin{problem}
Consider the action $f: G/G^{00} \times \FGen \to \FGen, (g,p) \mapsto g\cdot p$. Is it measurable? It is easy to see that $f$ is continuous for a fixed $g$ and measurable for a fixed $p$. In many situations this is sufficient for joint measurability of the map, but our case does not seem to be covered by any result in the literature.
\end{problem}

\section{Generic compact domination and the Ellis group conjecture}

\subsection{Baire-generic compact domination}

Let $G = G(\M)$  be a definably amenable NIP group, and let $M$ be a small model of $T$. Let $p\in S_{G}\left(\M\right)$
be a global type strongly $f$-generic over $M$. Let $\pi : G \to G/G^{00}$ be the canonical projection. It naturally lifts to a continuous map $\pi: S_G(\M) \to G/G^{00}$.
Fix a formula $\phi\left(x\right)\in L_G\left(\M\right)$, and we define
$U_{\phi\left(x\right)}=\left\{ g/G^{00}:g\cdot p\vdash\phi\left(x\right)\right\} \subseteq G/G^{00}$.
\begin{prop}\label{prop: UisConstructable}
The set $U = U_{\phi\left(x\right)}$ is a constructible subset of $G/G^{00}$
(namely, a Boolean combination of closed sets).\end{prop}
\begin{proof}
Note that $U=\pi\left(S\right)$ with $S=\left\{ g\in G:\phi\left(gx\right)\in p\right\} $.

As explained in Section \ref{sec: forking} we have
$S=\bigcup_{n<N}\left(A_{n}\land\neg B_{n+1}\right)$ for some $N< \omega$,
where:

\[
\mbox{Alt}_n (x_0, \ldots, x_{n-1}) = \bigwedge_{i < n - 1} \neg \left( \phi(g x_i) \leftrightarrow \phi (g x_{i+1}) \right)\mbox{,}
\]

\[
A_{n}=\{ g\in G:\exists x_{0}\ldots x_{n-1}(p^{\left(n\right)}|_{M}\left(x_{0},\ldots,x_{n-1}\right) \land 
 \mbox{Alt}_{n}(x_0, \ldots, x_{n-1})
\land \]
\[\land \phi\left(gx_{n-1}\right))\} \mbox{,}
\]
\[
B_{n}=\{ g\in G:\exists x_{0}\ldots x_{n-1}(p^{\left(n\right)}|_{M}\left(x_{0},\ldots,x_{n-1}\right)\land
\mbox{Alt}_n (x_0, \ldots, x_{n-1})
\land \]
\[\land \neg\phi\left(gx_{n-1}\right))\} \mbox{.}
\]

Note that $A_{n},B_{n}$ are type definable (over $M$ and the parameters of $\phi(x)$). Define

\[
A_n' := \left\{ g \in G : \exists h \in G \left ( g^{-1}h \in G^{00} \land h \in A_n \right ) \right\} \mbox{,}
\]
\[
B_n' := \left\{ g \in G : \exists h \in G \left ( g^{-1}h \in G^{00} \land h \in B_n \right ) \right\} \mbox{.}
\]

These are also type-definable sets. Let $S'=\bigcup_{n<N}\left(A_{n}'\land\neg B_{n+1}'\right)$. We check that $S'=S$.
Note:
\begin{enumerate}
	\item $S$ is $G^{00}$-invariant (because $p$ is),
	\item all of $A'_n, B'_n, S'$ are $G^{00}$-invariant (by definition),
	\item $A_n \subseteq A'_n, B_n \subseteq B'_n$.
\end{enumerate}

 First, if $g\in S'$, say $g\in A'_n \wedge \neg B'_{n+1}$, then there is $h\in G$ such that $hg^{-1} \in G^{00}$ and $h\in A_n$. As $g\in \neg B'_{n+1}$, also $h\in \neg B'_{n+1}$, and so $h \in \neg B_{n+1}$ (by (2) and (3)). Hence $h\in S$, and by (1)  also $g\in S$. So $S' \subseteq S$.

Assume that $g\in S \setminus S'$, and let $n < N$ be maximal for which there is $h\in gG^{00}$ such that $h\in A_n \wedge \neg B_{n+1}$. Then for a corresponding $h$, we still have $h \in S \setminus S'$ by (1) and (2). In particular, $h \notin A'_n \land \neg B'_{n+1}$. As $h \in A_n \subseteq A'_n$, necessarily $h \in B'_{n+1}$. This means that there is some $h' \in h G^{00} = g G^{00}$ such that $h' \in B_{n+1}$. As $h'$ is still in $S$ by (1), it follows that $h' \in A_m \land \neg B_{m+1}$ for some $m$, but by the definition of the $B_n$'s this is only possible if $m+1 > n+1$, contradicting the choice of $n$. Thus $S = S'$.

Now, we have $\pi(S')=\pi(S)=\bigcup_{n<N} \pi(A'_n)\wedge \neg \pi(B'_{n+1})$ since $A'_n$ and $B'_n$ are all $G^{00}$-invariant. As $\pi(A'_n)$, $\pi(B'_n)$ are closed, we conclude that $\pi(S)$ is constructible.
\end{proof}

Let $C :=\overline{G\cdot p} \subseteq S_G(\M)$, and we define $$E_{\phi\left(x\right)}=\left\{ \bar h\in G/G^{00}:\pi^{-1}\left(\bar h\right)\cap\phi\left(x\right)\cap C\neq\emptyset\mbox{ and }\pi^{-1}\left(\bar h\right)\cap\neg\phi\left(x\right)\cap C\neq\emptyset\right\} \mbox{.}$$

\begin{rem}\label{rem: boundary has empty interior}
Let $X$ be an arbitrary topological space, and let $Y \subseteq X$ be a constructible set.
Then the boundary $\partial Y$ has empty interior. 
\end{rem}
\begin{proof}
This is easily verified as $Y$ is a Boolean combination of closed sets, $\partial (Y_1 \cup Y_2) \subseteq \partial Y_1 \cup \partial Y_2$ for any sets $Y_1,Y_2 \subseteq X$, and $\partial Y$ has empty interior if $Y$ is either closed or open.
\end{proof}

\begin{thm}
\label{thm:CompactDomination}
(Baire-generic compact domination)
The set $E_{\phi\left(x\right)}$ is closed and has empty interior. In particular
it is meagre.
\end{thm}
\begin{proof}
We have $E_{\phi(x)}=\pi(\phi(x)\cap C)\cap \pi(\neg \phi(x)\cap C)$ and $\phi(x)\cap C$, $\neg \phi(x) \cap C$ are closed subsets of $S_G(\M)$, hence $E_{\phi(x)}$ is closed.

We may assume that $p$ concentrates on $G^{00}$, as replacing $p$ by $g \cdot p$ for some $g\in G(\M)$  does not change $C$, and thus does not change $E_{\phi(x)}$.

Let $\bar g \in E_{\phi(x)}$ be given, and let $V$ be an arbitrary open subset of $G/G^{00}$ containing $\bar g$. As the map $\pi$ is continuous, the set $S= \pi^{-1}(V)$ is an open subset of $S_G(\M)$. By the definition of $E_{\phi(x)}$, there must exist $q,q' \in C$ such that $\pi(q) = \pi(q') = \bar g$ and $q \in S \cap \phi(x), q' \in S \cap \neg \phi(x)$. As $C = \overline{G\cdot p}$, it follows that there are some $h, h' \in G(\M)$ such that $h \cdot p \in S \cap \phi(x)$ and $h' \cdot p \in S \cap \neg\phi(x)$. But then, as $p$ concentrates on $G^{00}$, $\pi(h) = \pi(h \cdot p) \in V \cap U$ and $\pi(h') = \pi(h' \cdot p) \in V \cap U^c$ (where $U = U_{\phi(x)}$ is as defined before Proposition \ref{prop: UisConstructable}). As $V$ was an arbitrary neigbourhood of $\bar g$, it follows that $\bar g \in \partial U$, hence $E_{\phi(x)} \subseteq \partial U$. 
By Proposition \ref{prop: UisConstructable}, $U$ is constructible. Hence $\partial U$ has  empty interior by Remark \ref{rem: boundary has empty interior}, and so $E_{\phi(x)}$ has empty interior as well.
\end{proof}

\subsection{Connected components in an expansion by externally definable sets}
Given a small model $M$ of $T$, an externally definable subset of $M$ is an intersection of an $L(\M)$-definable subset of $\M$ with $M$. One defines an expansion $M^{\ext}$ in a language $L'$ by adding a new predicate symbol for every externally definable subset of $M^n$, for all $n$. Recall that a global type $p\in S(\M)$ is finitely satisfiable in $M$ if $p$ lies in the topological closure of $M$, where $M$ is identified with its image in $S(\M)$ under the map sending $a\in M$ to the type $x=a$. There is a canonical bijection (even homeomorphism) between $S(M^{ext})$ and the subspace of types in $S(\M)$ finitely satisfiable in $M$. Recall also that a \emph{coheir} of a type $p\in S(M)$ is a type over a larger model $N$ which extends $p$ and is finitely satisfiable in $M$.

Let $T' = \Th_{L'}(M^{\ext})$. Note that automatically any quantifier-free $L'$-type over $M^{\ext}$ is definable (using $L'$-formulas). The following is a fundamental theorem of Shelah \cite{shelah2009dependent} (see also \cite{ExtDefI} for a refined version).
\begin{fact}\label{fac: Shelah expansion} Let $T$ be NIP, and let $M$ be a model of $T$. Then $T'$ eliminates quantifiers. 

It follows that $T'$ is NIP and that all ($L'$-)types over $M^{\ext}$ are definable.
\end{fact}
Assume now that $G$ is an $L$-definable group, and let $\M'$ be a monster model for $T'$ such that $\M \restriction L$ is a monster for $T$. In general there will be many new $L'$-definable subsets and subgroups of $G(\M')$ which are not $L$-definable. In \cite{ChePilSim} it is demonstrated however that many properties of definable groups are preserved when passing to $T'$.
\begin{fact} \label{fac: preservation in Mext} Let $T$ be NIP and let $M$ be a small model of $T$. Let $G$ be an $L$-definable group. 
\begin{enumerate}
\item If $G$ is definably amenable in the sense of $T$, then it is definably amenable in the sense of $T'$ as well.
\item The group $G^{00}(\M)$ computed in $T$ coincides with $G^{00}(\M')$ computed in $T'$.
\end{enumerate}
\end{fact}

In particular this implies that $G/G^{00}$ is the same group when computed in $T$ or in $T'$. Note also that the logic topology on $G/G^{00}$ computed in $T$ coincides with the logic topology computed in $T'$: any open set in the sense of $T$ is also open in the sense of $T'$ and both are compact Hausdorff topologies, therefore they must coincide.

\begin{rem}
In view of Remark \ref{rem: G/G^{00} is Polish},  if $L$ is countable then $G/G^{00}$ is still a Polish space with respect to the $L'$-induced logic topology.
\end{rem}

\subsection{Ellis group conjecture} \label{sec: Ellis group conj}

We recall the setting of definable topological dynamics and enveloping semigroups (originally from \cite[Section 4]{New4}, but we are following the notation from \cite{ChePilSim}).

Let $M_0$ be a small model of a theory $T$, and assume that all types over $M_0$ are definable. Then $G(M_0)$ acts on $S_G(M_0)$ by homeomorphisms, and the identity element $1$ has a dense orbit. The set $S_G(M_0)$ admits a natural semigroup structure $\cdot$ extending the group operation on $G(M_0)$ and continuous in the first coordinate: for $p,q \in S_G(M_0)$, $p \cdot q$ is $\tp(a\cdot b /M_0)$ where $b$ realizes $q$ and $a$ realizes the unique coheir of $p$ over $M_0b$. This semigroup is precisely the enveloping Ellis semigroup of $(S_G(M_0),G(M_0))$ (see e.g. \cite{Glasner1}). In particular left ideals of $(S_G(M_0),\cdot)$ are precisely the closed $G(M_0)$-invariant subflows of $G(M_0)\curvearrowright S_G(M_0)$, there is a minimal subflow $\mathcal{M}$ and there is an idempotent $u \in \mathcal{M}$. Moreover, $u \cdot \mathcal{M}$ is a subgroup of the semigroup $(S_G(M_0), \cdot)$ whose isomorphism type does not depend on the choice of $\mathcal{M}$ and $u \in \mathcal{M}$. It is called the Ellis group (attached to the data).
The quotient map from $G = G(\M)$ to $G/G^{00}_{M_0}$ factors through the tautological map $g \mapsto \tp(g/M_0)$ from $G$ to $S_G(M_0)$, and we let $\pi$ denote the resulting map from $S_G(M_0) \to G/G^{00}_{M_0}$. It is a surjective semigroup homomorphism, and for any minimal subflow $\mathcal{M}$ of $S_G(M_0)$ and $u \in \mathcal{M}$, the restriction of $\pi$ to $u \cdot \mathcal{M}$ is a surjective group homomorphism.

Now, let $T$ be NIP, and let $M$ be an arbitrary model. Then we consider $M_0 := M^{\ext}$, an expansion of $M$ by naming all externally definable subsets of $M^n$ for all $n \in \mathbb{N}$, in a new language $L'$ extending $L$. Then $T' = \Th_{L'}(M_0)$ is still NIP, and all $L'$-types over $M_0$ are definable (by Fact \ref{fac: Shelah expansion}), so the construction from the previous paragraph applies to $(S_G(M_0),G(M_0))$.
Let $\M'$ be a monster model for $T'$, so that $\M = \M' \restriction L$ is a monster model for $T$.
By Fact \ref{fac: preservation in Mext}, if $G(\M')$ is definably amenable in the sense of $T$, then it remains definably amenable in the sense of $T'$, and $G^{00}(\M) = G^{00}(\M')$ (the first one is computed in $T$ with respect to $L$-definable subgroups, while the second one is computed in $T'$ with respect to $L'$-definable subgroups). Newelski asked in \cite{New4} if the Ellis group was equal to $G/G^{00}$ for some nice classes of groups. Gismatullin, Penazzi and Pillay \cite{gismatullin2012some} show that this is not always the case for NIP groups ($SL_2(\mathbb R)$ is a counterexample). The following modified conjecture was then suggested by Pillay (see \cite{ChePilSim}):

\medskip
\noindent
{\bf Ellis group conjecture:} Suppose $G$ is a definably amenable NIP group. Then the restriction of $\pi : S_G(M_0) \to G/G^{00}$ to $u \cdot \mathcal{M}$ is an isomorphism, for some/any minimal subflow $\mathcal{M}$ of $S_G(M_0)$ and idempotent $u \in \mathcal{M}$ (i.e., $\pi$ is injective).

\begin{thm}
The Ellis group conjecture is true, i.e., $\pi: u \cdot \mathcal{M} \to G/G^{00}$ is an isomorphism.
\end{thm}
\begin{proof}
Fix notations as above. Throughout this proof, we work in $T'$. Let $p\in S_{G}\left(\M '\right)$ be strongly $f$-generic
over $M_0$. Let $C :=\overline{G\cdot p}$, and let $V :=\left\{ p|_{M_0}:p\in C\right\} $.
Note that $V$ is a subflow of $G\left(M_0\right)\curvearrowright S_{G}\left(M_0\right)$:
it is closed as a continuous image of a compact set $C$ into a Hausdorff
space, and it is $G\left(M_0\right)$-invariant as $C$ is $G\left(\M'\right)$-invariant.
Let $\mathcal{M}$ be a minimal subflow of $V$. It has to be of the
form $\overline{G\left(M_0\right)\cdot\left(p'|_{M_0}\right)}$ for some
$p'\in C$. So replacing $p$ by $p'$ (which is still strongly $f$-generic
over $M_0$) we may assume that $\mathcal{M}=\overline{G\left(M_0\right)\cdot\left(p|_{M_0}\right)}$
is minimal.

Let $u\in\mathcal{M}$ be an idempotent. We will show that if $p_{1},p_{2}\in u\cdot\mathcal{M}$
and $\pi\left(p_{1}\right)=\pi\left(p_{2}\right)$, i.e., they determine
the same coset of $G^{00}$, then there is some $r\in\mathcal{M}$
such that $r\cdot p_{1}=r\cdot p_{2}$. By the general theory of Ellis
semigroups (see e.g. \cite[Proposition 2.5(5)]{Glasner1}) this will imply that $p_{1}=p_{2}$,
as wanted.

Let $\mathcal{F}$ be the filter of comeagre subsets of $G/G^{00}$, and let $\mathcal{F}'$ be
some ultrafilter extending it. Let $q_{1},q_{2}\in C$ be some global
types extending $p_{1},p_{2}$ respectively. For each $\bar g\in G/G^{00}$,
let $r_{\bar g}\in S_{G}\left(M_0\right)$ be a type in $\mathcal{M}$ with
$\pi\left(r_{\bar g}\right)= \bar g$. Let $r=\lim_{\mathcal{F}'}r_{\bar g}$. Note
that $r\in\mathcal{M}$.

Let $\M^*\succ\M'$ be a larger monster of $T'$. Let $a_{i}\in\M^*$ be such
that $a_{i}\models q_{i}$ for $i=1,2$. For each $\bar g\in G/G^{00}$
let $r_{\bar g}'$ be the unique coheir of $r_{\bar g}$ over $\M^*$,
and let $b_{\bar g}\models r_{\bar g}'|_{\M'a_1 a_2}$. Finally, let  $r'= \lim_{\mathcal{F}'} r'_{\bar g}$,
the unique coheir of $r$ over $\M^*$, and let $b\in\M^*$ realize $r'|_{\M'a_1a_2}$.

\smallskip
\note{Claim 1.} $\lim_{\mathcal{F}'}\tp\left(b_{\bar g}\cdot a_{i}/\M' \right)=\tp\left(b\cdot a_{i}/\M' \right)$
for $i=1,2$.

This follows by left continuity of the semigroup operation, but we give
the details. Let $\phi\left(x\right)\in L'\left(\M'\right)$ be arbitrary,
and let $a'_i \in \M'$ be such that $a'_i \models q_{i}|_{N}$, where $N \succeq M_0$ is some small model over
which $\phi\left(x\right)$ is defined. Then we have:
\begin{equation*}
\begin{split}
\phi(x) \in \lim_{\mathcal{F}'} \left( \tp( b_{\bar g} \cdot a_i / \M') \right) 
\Leftrightarrow 
\left\{ \bar g \in G/G^{00} : \models \phi(b_{\bar g} \cdot a_i) \right\} \in \mathcal{F}'
\Leftrightarrow
\\
\Leftrightarrow 
\left\{ \bar g \in G/G^{00} : \models \phi(b_{\bar g} \cdot a'_i) \right \} \in \mathcal{F}'
\Leftrightarrow
\phi(x \cdot a'_i) \in \lim_{\mathcal{F}'} \left( \tp(b_{\bar g}/ \M') \right) \subseteq r'
\Leftrightarrow
\\
\Leftrightarrow
\phi(x \cdot a_i) \in r'
\Leftrightarrow
\models \phi( b \cdot a_i).
\end{split}
\end{equation*}

The second equivalence is by $M_0$-invariance of $r'_{\bar g}$, and the fourth one is by $M_0$-invariance of $r'$.

\smallskip
\note{Claim 2.} $r\cdot p_{1}=r\cdot p_{2}$.

Assume not, say there exists some $\phi\left(x\right)\in L' \left(\M'\right)$
such that $\phi\left(x\right)\in r\cdot p_{1}$, $\neg\phi\left(x\right)\in r\cdot p_{2}$,
so $\models\phi\left(b\cdot a_{1}\right)\land\neg\phi\left(b\cdot a_{2}\right)$ (according to the choice of $a_1,a_2,b$ and the definition of the semigroup operation on $S_G(M_0)$).
We may assume that both $q_{1}$ and $q_{2}$ concentrate on $G^{00}$.
By Claim 1 we have $\left\{ \bar g\in G/G^{00}:\models\phi\left(b_{\bar g}\cdot a_{1}\right)\land\neg\phi\left(b_{\bar g}\cdot a_{2}\right)\right\} \in\mathcal{F}'$.
As $E_{\phi\left(x\right)}\subseteq G/G^{00}$ is meagre by Theorem
\ref{thm:CompactDomination}, we have $\left(E_{\phi\left(x\right)}\right)^c \in\mathcal{F}'$,
and so there is some $\bar{g} \notin E_{\phi\left(x\right)}$ such that $\models\phi\left(b_{\bar g}\cdot a_{1}\right)\land\neg\phi\left(b_{\bar g}\cdot a_{2}\right)$.

For an arbitrary open set $V \subseteq G/G^{00}$ containing $\bar g$, we can choose $h \in G(\M')$ such that $\pi(h) \in V$ and $\phi\left(h \cdot a_{1}\right)\land\neg\phi\left(h \cdot a_{2}\right)$ holds. Indeed, let $S = \pi^{-1}(V) \subseteq S_G(\M')$, which is open by continuity of $\pi$. Then there is an $L'(\M')$-definable set $\psi(x) \subseteq S$ such that $\pi(\psi(x)) \subseteq V$ and $\models \psi(b_{\bar g})$. By finite satisfiability of $r'_{\bar g}$, take $h \in G(\M')$ satisfying $\phi\left(x \cdot a_{1}\right)\land\neg\phi\left(x \cdot a_{2}\right) \land \psi(x)$. As $\bar g \notin E_{\phi(x)}$ and $E_{\phi(x)}$ is closed by Theorem \ref{thm:CompactDomination}, we find such an $h$ with $\pi(h) \notin E_{\phi(x)}$.

Note that $\pi\left(h \cdot a_{1}\right)=\pi\left(h \right)=\pi\left(h \cdot a_{2}\right)$
as $q_{1},q_{2}$ concentrate on $G^{00}$, and that $\tp(h \cdot a_1 / \M') = h \cdot q_1 \in C, \tp(h \cdot a_2 / \M') = h \cdot q_2 \in C$. It follows that $\pi(h) \in E_{\phi(x)}$ --- a contradiction.
\end{proof}

\begin{cor}
In a definably amenable NIP group, the Ellis group of the dynamical system $(S_G(M^{ext}), G(M))$ is independent of the model $M$.
\end{cor}

\section{Further remarks}
\subsection{Left vs. right actions}\label{sec_leftright}

Until now, we have only considered the action of the group $G$ on itself by left-translations. One could also let $G$ act on the right and define analogous notions of right-f-generic, right-invariant measure etc. In a stable group, a type is left-generic if and only if it is right-generic so we obtain nothing new. However, in general, left and right notions may differ.

We start with an example of a left-invariant measure which is not right-invariant.

\begin{ex}
Let $G=(\mathbb R,+)\rtimes \{\pm 1\}$, where the two-element group $\{\pm 1\}$ acts on $\mathbb R$ by multiplication. Consider $G$ as a group defined in a model $R$ of RCF with universe $R \times \{-1,1\}$ and multiplication defined by $(x_0,\epsilon_0)\cdot (x_1,\epsilon_1)=(x_0+\epsilon_0x_1,\epsilon_0\epsilon_1)$. Let $p^{+}_{+\infty}(x,y)$ be the type whose restriction to $x$ is the type at $+\infty$ and which implies $y=1$. Define similarly $p^{-}_{-\infty}$. Then $\mu = \frac 1 2 \left ( p^{+}_{+\infty} + p^{-}_{-\infty} \right )$ is left-invariant, but not right-invariant.
\end{ex}

However, some things can be said.

\begin{lem}\label{lem_biinv}
Let $G = G(\M)$ be definably amenable, then there is always a measure on $G$ which is both left and right invariant.
\end{lem}
\begin{proof}
Let $\mu$ be a left invariant measure on $G$ which is also invariant over some small model $M$ (always exists in a definably amenable NIP group, e.g. by \cite[Lemma 5.8]{NIP2}).


Let $\mu^{-1}$ be defined by $\mu^{-1}(X) := \mu(X^{-1})$ for every definable set $X \subseteq G$, where $X^{-1} := \{a^{-1} : a \in X \}$. Then $\mu^{-1}$ is also a measure,  $M$-invariant (as $\mu^{-1}(\sigma(X)) = \mu(\sigma(X)^{-1}) = \mu(\sigma(X^{-1})) = \mu(X^{-1}) = \mu^{-1}(X)$  for any automorphism $\sigma \in \Aut(\M/M)$) and right invariant (as $\mu^{-1}(X \cdot g ) = \mu (g^{-1} \cdot X^{-1}) = \mu (X^{-1}) = \mu^{-1}(X)$ for any $g\in G$).

For any $\phi(x) \in L_G(\M)$, we define $\nu(\phi(x)) := \mu \otimes \mu^{-1} (\phi(x \cdot u))$. That is, for any definable set $X \subseteq G$ and a model $N$ containing $M$ and such that $X$ is $N$-definable, we have $\nu(X) = \int_{S_G(N)} f_X d\mu^{-1}$, where for every $q \in S_G(N)$, $f_X(q) = \mu(X \cdot h^{-1})$ for some/any $h \models q$ (well-defined by $M$-invariance of $\mu$, see Section \ref{sec: Keisler measures}). Then $\nu$ is an $M$-invariant measure, and given any $g \in G$ and  $N$ such that $g$ and $X$ are $N$-definable, for any $q \in S_G(N)$ and $h \models q$ we have:

\begin{enumerate}
	\item $f_{g \cdot X} (q) = \mu((g \cdot X) \cdot h^{-1}) = \mu(g \cdot (X \cdot h^{-1})) = \mu(X \cdot h^{-1}) = f_X(q)$, by left invariance of $\mu$.
	\item $f_{X \cdot g}(q) = \mu( (X \cdot g) \cdot h^{-1}) = f_{X}( q \cdot g^{-1})$, and $\int_{S_G(N)} f_X(q) d\mu^{-1} = \int_{S_G(N)} f_X(q \cdot g^{-1}) d (\mu^{-1} \cdot g) = \int_{S_G(N)} f_X(q \cdot g^{-1}) d (\mu^{-1})$ as $\mu^{-1} = \mu^{-1} \cdot g$ by right invariance.
\end{enumerate}
Hence $\nu$ is both left and right invariant.
\end{proof}

\begin{prop}
Let $G$ be definably amenable and let $\phi(x)\in L_G(\M)$. If $\phi(x)$ is left-generic, then it is right-f-generic.
\end{prop}
\begin{proof}
By the previous lemma, let $\mu(x)$ be a left and right invariant measure on $G$. Then as $\phi(x)$ is left-generic, we must have $\mu(\phi(x))>0$. But as $\mu$ is also right-invariant, this implies that $\phi(x)$ is right-f-generic (by the ``right hand side'' counterpart of Proposition \ref{lem_PosGInvMeasImpliesFGen}).
\end{proof}

As the following example shows, no other implication holds.

\begin{ex}
Let $R$ be a saturated real closed field and let $G=(R^{2},+)\rtimes SO(2)$ with the canonical action, seen as a definable group in $R$. For $0<a<1$ let $C_a\subset R^2$ be the angular region defined by $\{(x,y) : x\geq 0 ~ \&~  |y|\leq a\cdot x\}$. Finally, let $X_a = C_a \times SO(2) \subseteq G$. 

Note that any two translates of $C_a$ intersect. Hence any two right translates of $X_a$ intersect: let $g=(x_g,\sigma_g) \in G$, then $X_a\cdot g = \bigcup_{\tau\in SO(2)} (C_a+\tau(x_g)) \times \{\tau\cdot \sigma_g\}$; hence $X_a\cdot g \cap X_a$ is non-empty and in fact has surjective projection on $SO(2)$. This shows that $X_a$ is right-f-generic.

On the other hand, multiplying $X_a$ on the left has the effect of turning it: $g\cdot X_a= (x_g + \sigma_g(C_a))\times SO(2)$. If $a$ is infinitesimal, then there are infinitely many pairwise disjoint left-translates of $X_a$, hence $X_a$ is not left-f-generic. If however $a$ is not infinitesimal, then we can cover $R^2$ by finitely many $SO(2)$-conjugates of $C_a$, and hence cover $G$ by finitely many left-translates of $X_a$.

We conclude that if $a$ is infinitesimal, then $X_a$ is right-f-generic but not left-f-generic, and if $a$ is not infinitesimal, then $X_a$ is left-generic but not right-generic.
\end{ex}

\subsection{Actions on definable homogeneous spaces}

While the theory above was developed for the action of a definably amenable group $G$ on $S_G(\M)$, we remark that (with obvious rephrasements) it works just as well for a definably amenable group $G = G(\M)$ acting on $S_X(\M)$ for $X$ a definable homogeneous $G$-space (i.e. $X$ is a definable set, the graph of the action map $G\times X \to X$ is definable and the action is transitive). We show that given a definable homogeneous space $X$ for a definably amenable group $G$, every $G$-invariant measure on $G$ pushes forward to a $G$-invariant measure on $X$, and conversely any $G$-invariant measure on $X$ lifts to a $G$-invariant measure on $G$, possibly non-uniquely.

\begin{lem}\label{lem: extending measure by an ideal}
Let $B_0 \subseteq \Def(\M)$ be a Boolean algebra and let $I \subseteq \Def(\M)$ be an ideal such that $I \cap B_0$ is contained in the zero-ideal of $\nu_0$, a measure on $B_0$.

Let $B$ be the collection of all sets $U \in \Def(\M)$ for which there is some $V \in B_0$ such that $U \triangle V \in I$. Then $B$ is a Boolean algebra with $B_0, I \subseteq B$. Moreover, $\nu_0$ extends to a global measure $\nu$ on $\Def(\M)$ such that all sets from $I$ have $\nu$-measure $0$.
\end{lem}
\begin{proof}

It can be checked straightforwardly that $B$ is a Boolean algebra containing $B_0$ and $I$. Now for $U \in B$, let $\nu'(U) = \nu_0(V)$ where $V$ is some set in $B_0$ with $U \triangle V \in I$.

\begin{enumerate} 
\item $\nu'$ is well-defined. If we have some $V' \in B_0$ with $U \triangle V' \in I$, then $V \triangle V' \subseteq (U \triangle V) \cup (U \triangle V') \in I$, so $V \triangle V' \in I$; but by assumption this implies that $\nu_0(V \triangle V') = 0$, so $\nu_0(V) = \nu_0(V')$.
\item $\nu'$ is a measure on $B$ extending $\nu_0$. Given $U_i \in B, i\leq 2$, let $V_i \in B_0$ be such that $U_i \triangle V_i \in I, i\leq 2$. Then $\nu'( U_1 \cup U_2) = \nu( V_1 \cup V_2) = \nu(V_1) + \nu(V_2) - \nu(V_1 \cap V_2) = \nu'(U_1) + \nu'(U_2) - \nu'(U_1 \cap U_2)$, as wanted.
\end{enumerate}

Now $\nu'$ extends to a global measure $\nu$ by compactness, see e.g. \cite[Lemma 7.3]{SimBook}.
\end{proof}

\begin{prop} \label{prop: measures push pull}Let $X$ be a definable homogeneous $G$-space, and let $x_0$ be an arbitrary point in $X$.

\begin{enumerate}
\item Let $\tilde{\mu}$ be a measure on $G$. For every definable subset $\phi(x)$ of $X$, let $\mu(\phi(x)) = \tilde{\mu}(\phi(u \cdot x_0))$. Then $\mu$ is a measure on $X$. Moreover, if $\tilde{\mu}$ is $G$-invariant, then $\mu$ is $G$-invariant as well. If $\tilde{\mu}$ is also right-invariant, then $\mu$ does not depend on the choice of $x_0$.

\item Assume moreover that $G$ is definably amenable, NIP. Let $\mu$ be a $G$-invariant measure on $X$. Then there is some (possibly non-unique) $G$-invariant measure $\tilde{\mu}$ on $G$ such that the procedure from (1) induces $\mu$.
\end{enumerate} 
\end{prop}
\begin{proof}
(1) It is clearly a measure as $\mu(\emptyset) = \tilde{\mu}(\emptyset)$, $\mu(X) = \tilde{\mu}(G)$ and if $\phi_i(x), i < n$ are disjoint subsets of $X$, then $\phi_i(u\cdot x_0), i<n$ are disjoint subsets of $G$. If $\tilde{\mu}$ is $G$-invariant then for any $g\in G$ we have $\mu( \phi(g^{-1} \cdot x)) = \tilde{\mu}( \phi(g^{-1} \cdot u \cdot x_0)) = \tilde{\mu}(\phi(u \cdot x_0)) = \mu(\phi(x))$.

Finally, assume that $\tilde{\mu}$ is also right invariant. Let $x_1 \in X$ and $\phi(x)$ be arbitrary, then by transitivity of the action there is some $g \in G$ such that $x_1 = g \cdot x_0$. We have $\tilde{\mu}(\phi(u \cdot x_1)) = \tilde{\mu}(\phi(u \cdot (g \cdot x_0))) = \tilde{\mu}(\phi((u \cdot g) \cdot x_0 )) = \tilde{\mu}(\phi(u \cdot x_0) \cdot g^{-1}) = \tilde{\mu}(\phi(u \cdot x_0))$, as wanted.

(2) Now let $\mu$ be a $G$-invariant measure on $X$, and fix $x_0 \in X$. Let $B_0 \subseteq \Def_G(\M)$ be the family of subsets of $G$ of the form $\{ g \in G: g \cdot x_0 \in Y\}$, where $Y$ is a definable subset of $X$. For $U \in B_0$, define $\nu_0(U) = \mu(Y)$. The following can be easily verified using that $\mu$ is a $G$-invariant measure:

\textbf{Claim.}
The family $B_0$ is a Boolean algebra closed under $G$-translates and $\nu_0$ is a $G$-invariant measure on $B_0$.

Next, let $I \subseteq \Def_G(\M)$ be the collection of all non-$f$-generic definable subsets of $G$. We know by Corollary \ref{prop_ideal} that it is an ideal. As in Proposition \ref{lem_PosGInvMeasImpliesFGen}, $B_0 \cap I$ is contained in the zero-ideal of $\nu_0$. Then, applying Lemma \ref{lem: extending measure by an ideal}, we obtain a global measure $\nu$ on $\Def_G(\M)$ extending $\nu_0$ and such that all types in its support are $f$-generic. Note that $\nu$ is $G^{00}$-invariant: for any $\phi(x) \in L(\M)$ and $\varepsilon > 0$ there are some $p_0, \ldots, p_{n-1} \in S(\nu)$ such that for any $g \in G$,  $\nu(g \phi(x)) \approx^{\varepsilon} \Av(p_0, \ldots, p_{n-1}; g \phi(x))$ (by Fact \ref{fac: measure is average of types}), and each $p_i$ is $G^{00}$-invariant (by Proposition \ref{prop:f-gen iff G00-inv}). Consider the map $f_{\phi}: G/G^{00} \to \mathbb{R}, \overline{g} \mapsto \nu(g \phi(x))$. It is well-defined  and $h_0$-measurable (using an argument as in the proof of Lemma \ref{lem_measurability}). Finally, we define $\tilde{\mu}(\phi(x)) = \int_{g \in G/G^{00}} f_\phi(g) \textrm{d} h_0$. It is easy to check that $\tilde{\mu}$ is a $G$-invariant measure on $\Def_G(\M)$ (and that the procedure from (1) applied to $\tilde{\mu}$ returns $\mu$).
\end{proof}




\bibliography{common}

\end{document}